\documentclass[12pt]{amsart}
\usepackage[bookmarks=false]
 {hyperref}
\hypersetup{
colorlinks=true,
urlcolor=black,
citecolor=blue,
linkcolor=blue,
}
\usepackage{amsthm}
\usepackage{amssymb}
\usepackage{geometry}
\usepackage{enumerate}
\usepackage{setspace}
\usepackage{upgreek}
\usepackage{pgfkeys}
\usepackage{tikz}
\usepackage{tikz-cd}
\usetikzlibrary{matrix, arrows}
\newtheorem{thm}{Theorem}[section]
\newtheorem{prop}[thm]{Proposition}
\newtheorem{lem}[thm]{Lemma}

\theoremstyle{definition}
\newtheorem{definition}[thm]{Definition}

\theoremstyle{remark}
\newtheorem{remark}[thm]{Remark}

\numberwithin{equation}{section}

\newcommand{\gal}{\mbox{gal}}
\newcommand{\Gal}{\mbox{Gal}}
\newcommand{\cl}{\mbox{cl}}
\newcommand{\Cl}{\mbox{Cl}}
\newcommand{\Hom}{\mbox{Hom}}
\newcommand{\Map}{\mbox{Map}}
\newcommand{\lift}{\overline}
\newcommand{\cO}{\mathcal{O}}

\begin{document}

\large

\title{Realizable Classes and Embedding Problems}

\author{Cindy (Sin Yi) Tsang}
\address{Department of Mathematics, University of California, Santa Barbara}
\email{cindytsy@math.ucsb.edu}
\urladdr{http://sites.google.com/site/cindysinyitsang} 

\date{\today}

\begin{abstract}Let $K$ be a number field with ring of integers $\cO_K$ and let $G$ be a finite group. Given a $G$-Galois $K$-algebra $K_h$, let $\cO_h$ denote its ring of integers. If $K_h/K$ is tame, then a classical theorem of E. Noether implies that $\cO_h$ is locally free over $\cO_KG$ and hence defines a class in the locally free class group of $\cO_KG$. Write $R(\cO_KG)$ for the set of all such classes. For $G$ abelian, by combining the work of J. Brinkhuis and L. McCulloh, we will prove that the structure of $R(\cO_KG)$ is related to the study of embedding problems.
\end{abstract}

\maketitle

\tableofcontents

\section{Introduction and Preliminaries}\label{intro}

Let $K$ be a number field with ring of integers $\cO_K$ and let $G$ \mbox{be a finite group.} Let $K^c$ denote a fixed algebraic closure of $K$ and let $\Omega_K:=\Gal(K^c/K)$ act trivially on $G$. Then, the set of all isomorphism classes of $G$-Galois $K$-algebras (see Subsection~\ref{galois} for a brief review) is in bijective correspondence with the pointed Galois cohomology set $H^1(\Omega_K,G)$. Given an element $h\in H^1(\Omega_K,G)$, we will write $K_h$ for a Galois algebra representative of $h$ and $\cO_h$ for the ring of integers in $K_h$.

If $K_h/K$ is tame, then a classical theorem of E. Noether implies that $\cO_h$ is locally free over $\cO_KG$ and hence defines a class $\cl(\cO_h)$ in the locally free class group $\Cl(\cO_KG)$ of $\cO_KG$. Such a class in $\Cl(\cO_KG)$ is called \emph{realizable}. Set
\[
H_t^1(\Omega_K,G):=\{h\in H^1(\Omega_K,G)\mid K_h/K\mbox{ is tame}\}.
\]
Then, there is a natural map
\begin{equation}\label{gal}
\gal:H^1_t(\Omega_K,G)\longrightarrow \mbox{Cl}(\cO_KG);\hspace{1em}\gal(h):=\cl(\cO_h)
\end{equation}
whose image is equal to the set
\[
R(\cO_KG):=\{\cl(\cO_h):h\in H^1_t(\Omega_K,G)\}
\]
of all realizable classes in $\Cl(\cO_KG)$. 

In \cite{B}, J. Brinkhuis related the study of realizable classes to that of embedding problems as follows. Let $K/k$ be a Galois subextension of $K$ and define $\Sigma:=\Gal(K/k)$. 

\begin{definition}\label{embedproblem}Given a group extension $E$ of $\Sigma$ by $G$, say
\[
\begin{tikzcd}[column sep=1cm, row sep=1.5cm]
E:\hspace{0.5cm}
1 \arrow{r} &
G \arrow{r} &
\Gamma \arrow{r} &
\Sigma\arrow{r} &
1,
\end{tikzcd}
\]
a \emph{solution} to the embedding problem $(K/k,G,E)$ is a Galois extension $N/K$ such that $N/k$ is also Galois and that there are isomorphisms $\Gal(N/K)\simeq G$ and $\Gal(N/k)\simeq\Gamma$ making the diagram
\[
\begin{tikzpicture}[baseline=(current bounding box.center)]
\node at (-1,3) [name=A] {$1$};
\node at (13,3) [name=B] {$1$};
\node at (-1,1) [name=C] {$1$};
\node at (13,1) [name=D] {$1$};
\node at (2,3) [name=1] {$\Gal(N/K)$};
\node at (6,3) [name=2] {$\Gal(N/k)$};
\node at (10,3) [name=3] {$\Gal(K/k)$};
\node at (2,1) [name=4] {$G$};
\node at (6,1) [name=5] {$\Gamma$};
\node at (10,1) [name=6] {$\Sigma$};
\path[dashed,font=\normalsize]
(2) edge node[auto] {} (3);
\path[->]
(A) edge node[auto] {} (1)
(C) edge node[auto] {} (4)
(3) edge node[auto] {} (B)
(6) edge node[auto] {} (D)
(1) edge node[auto] {} (2)
(2) edge node[auto] {} (3)
(4) edge node[auto] {} (5)
(5) edge node[auto] {} (6);
\path[<->]
(1) edge node[auto] {$\simeq$} (4)
(2) edge node[auto] {$\simeq$} (5);
\draw[double equal sign distance]
(3) -- (6);
\end{tikzpicture}
\]
commute. If in addition $N/K$ is tame, then we will call $N/K$ a \emph{tame solution}.
\end{definition}

Suppose now that $G$ is abelian. Then, the pointed set $H^1(\Omega_K,G)$ is equal to $\Hom(\Omega_K,G)$ and hence has a group structure. Let $K^t$ denote the maximal tamely ramified extension of $K$ contained in $K^c$ and let $\Omega_K^t:=\Gal(K^t/K)$ act trivially on $G$. Then, the subset $H_t^1(\Omega_K,G)$ may be naturally identified with $\Hom(\Omega_K^t,G)$ (see Remark~\ref{tamesubgp} below) and in particular is a subgroup of $H^1(\Omega_K,G)$. We note that the map $\gal$ is not a homomorphism in general, but is only \emph{weakly multiplicative} in the following sense. Let $M_K$ denote the set of primes in $\cO_K$. For $h\in \Hom(\Omega_K^t,G)$, define
\begin{equation}\label{d(h)}
d(h):=\{v\in M_K\mid K_h/K\mbox{ is ramified at $v$}\}.
 \end{equation}
Then, for all $h_1,h_2\in\Hom(\Omega_K^t,G)$, we have
\begin{equation}\label{galweak}
\gal(h_1h_2)=\gal(h_1)\gal(h_2)\hspace{1cm}\mbox{whenever }d(h_1)\cap d(h_2)=\emptyset.
\end{equation}
This weak multiplicativity of $\gal$ was first proved by Brinkhuis in \cite[Proposition 3.10]{B} and it also follows from the result \cite[Theorem 6.7]{M} of L. McCulloh.

In what follows, we will fix a left $\Sigma$-module structure of $G$. In \cite[Theorem 5.1]{B}, Brinkhuis constructed a commutative diagram
\begin{equation}\label{bd}
\begin{tikzcd}[column sep=1.5cm, row sep=2.25cm]
H^1(\Gal(K^t/k),G)\arrow{r}{\mbox{res}}
& \Hom(\Omega_K^t,G)^\Sigma\arrow{r}[font=\large]{tr}\arrow[swap]{d}{\mbox{gal}}
& H^2(\Sigma,G)\arrow{d}[font=\large]{i^*}\\
&\mbox{Cl}(\cO_KG)^\Sigma\arrow{r}[font=\large]{\xi}
&H^2(\Sigma,(\cO_KG)^\times)
\end{tikzcd}
\end{equation}
(see Section~\ref{basicdiagram} below for the construction and notation), where the top row is \par\noindent exact and all of the maps except possibly $\gal$ are homomorphisms. We will refer to diagram (\ref{bd}) as the \emph{basic diagram}.

\begin{remark}\label{bdmodified}Diagram (\ref{bd}) is a slightly modified and abridged version of the basic diagram constructed by Brinkhuis in \cite[Theorem 5.1]{B}. For example, the Picard group of $\mathcal{O}_KG$ was used in place of the locally free class group of $\cO_KG$, but these two groups are canonically isomorphic when $G$ is abelian (see \cite[Theorem 55.26]{CR2}, for example). In Theorem~\ref{commutes}, we will give a proof of the facts that (\ref{bd}) commutes, that the top row is exact, and that all of the maps in (\ref{bd}) except possibly $\gal$ are homomorphisms.
\end{remark}

The commutativity of (\ref{bd}) relates the study of realizable classes to that of embedding problems as follows. Let $h\in\Hom(\Omega_K^t,G)^\Sigma$ be given and assume that $h$ is surjective, in which case $K_h$ is isomorphic to $L:=(K^t)^{\ker(h)}$. We will show in Proposition~\ref{Brelation} below that $L$ is a tame solution to the embedding problem $(K/k,G,E_h)$, where the equivalence class of $E_h$ is determined by $tr(h)$. Now, suppose that $i^*$ is injective (as is shown in \cite[Theorem 4.1]{B2}, this is so when $K$ is a C.M. field and when $G$ or $\Sigma$ has odd order). If $tr(h)\neq 1$ (which corresponds to $E_h$ being non-split), then $\cl(\cO_h)\neq 1$ as well because (\ref{bd}) commutes and $\xi$ is a homomorphism.
  
We continue to assume that $G$ is abelian. In \cite[Theorem 6.17 and Corollary 6.20]{M}, McCulloh gave a characterization of the set $R(\cO_KG)$ and showed that it is a subgroup of $\Cl(\cO_KG)$. It is natural to ask whether the group structure of $R(\cO_KG)$ is also related to the study of embedding problems. More precisely, consider the subsets
\begin{align}\label{Srealizable}
R_\Sigma(\cO_KG)&:=\{\cl(\cO_h):h\in \Hom(\Omega_K^t,G)^\Sigma\}\\\notag
R_s(\cO_KG)&:=\{\cl(\cO_h):h\in \Hom(\Omega_K^t,G)^\Sigma\mbox{ and }tr(h)=1\}
\end{align}
of $R(\cO_KG)$. The classes in $R_\Sigma(\cO_KG)$ are called \emph{$\Sigma$-realizable}. We want to determine whether the sets above are subgroups of $R(\cO_KG)$, and if so, whether the group structure of $R_\Sigma(\cO_KG)/R_s(\cO_KG)$ is related to that of $H^2(\Sigma,G)$.

By combining the work Brinkhuis and McCulloh, we will prove the following partial result. Given a set $V$ of primes in $\cO_K$, define

\[
\Hom(\Omega_K^t,G)^\Sigma_V:=\{h\in \Hom(\Omega_K^t,G)^\Sigma\mid K_h/K\mbox{ is unramified at all $v\in V$}\}
\]
and
\begin{align*}
R_\Sigma(\cO_KG)_V&:=\{\cl(\cO_h):h\in \Hom(\Omega_K^t,G)_V^\Sigma\};\\
R_s(\cO_KG)_V&:=\{\cl(\cO_h):h\in \Hom(\Omega_K^t,G)_V^\Sigma\mbox{ with }tr(h)=1\}.
\end{align*}
Write $\exp(G)$ for the exponent of the group $G$.

\begin{thm}\label{thm}Let $K/k$ be a Galois extension of number fields and let $G$ be a finite abelian group. Let $\Sigma:=\Gal(K/k)$ act trivially on $G$ \mbox{on the left and} let $V=V_K$ denote the set of primes in $\cO_K$ which are ramified over $k$. Assume also that $k$ contains all $\exp(G)$-th roots of unity.
\begin{enumerate}[(a)]
\item The sets $R_\Sigma(\cO_KG)_V$ and $R_s(\cO_KG)_V$ are subgroups of $\Cl(\cO_KG)$. Further- more, given $h\in \Hom(\Omega_K^t,G)^\Sigma_V$ and a finite set $T$ of primes in $\cO_K$, there exists $h'\in \Hom(\Omega_K^t,G)^\Sigma_V$ such that
\begin{enumerate}[(1)]
\item $K_{h'}/K$ is a field extension;
\item $K_{h'}/K$ is unramified at all $v\in T$;
\item $\cl(\cO_{h'})=\cl(\cO_{h})$;
\item $tr(h')=tr(h)$.
\end{enumerate}
\item The natural surjective map
\[
\upphi:tr(\Hom(\Omega_K^t,G)^\Sigma_V)\longrightarrow \frac{R_\Sigma(\cO_KG)_V}{R_s(\cO_KG)_V};
\hspace{1em}
\upphi(tr(h)):=\cl(\cO_h) R_s(\cO_KG)_V,
\]
where $h\in \Hom(\Omega_K^t,G)_V^\Sigma$, is well-defined and is a homomorphism. Moreover, if $i^*$ is injective, then $\upphi$ is an isomorphism.
\end{enumerate}
\end{thm}

\begin{remark}If $G$ has odd order, then by Hilbert's formula (see \cite[Chapter IV Proposition 4]{S}, for example), for each $h\in H^1(\Omega_K,G)$ there exists a fractional ideal $A_h$ in $K_h$ whose square is the inverse different ideal of $K_h/K$. If $K_h/K$ is tame, then $A_h$ is locally free over $\cO_KG$ by \cite[Theorem 1]{U} or \cite[Theorem 1 in Section 2]{E}, and hence defines a class $\cl(A_h)$ in $\Cl(\cO_KG)$. If $G$ is abelian in addition, then by adapting the techniques developed \mbox{by McCulloh in \cite{M}, the} author has shown in \cite[Theorems 1.2 (b) and 1.3]{T} that the map

\[
\gal_A:H^1_t(\Omega_K,G)\longrightarrow\Cl(\cO_KG);\hspace{1em}\gal_A(h):=\cl(A_h)
\]
is weakly multiplicative in the sense of (\ref{galweak}) and that
\[
\mathcal{A}^t(\cO_KG):=\{\cl(A_h):h\in H^1_t(\Omega_K,G)\}
\]
is a subgroup of $\Cl(\cO_KG)$. With the additional assumption that $G$ has odd order, the author shows in \cite[Theorems 1.4.4 and 1.4.5]{thesis} that Theorems~\ref{commutes} and~\ref{thm} still hold when one replaces $\gal$ and $\cO_h$ by $\gal_A$ and $A_h$, respectively. The proofs of the corresponding statements are essentially the same.
\end{remark}

In the subsequent subsections, we will give a brief review of locally free class groups and Galois algebras. From Section~\ref{basicdiagram} onwards, we will assume that $G$ is abelian. In Section~\ref{basicdiagram}, we will construct the basic diagram and show that it commutes. In Section~\ref{R}, we will recall the necessary \mbox{definitions in order} to state the characterization of $R(\cO_KG)$ from \cite[Theorem 6.17]{M}. In Section~\ref{SR}, we will modify this characterization and prove Theorem~\ref{thm}.

\subsection{Notation and Conventions}\label{notation} Throughout this paper $G$ is a fixed finite group. We will also use the convention that all of the homomorphisms in the cohomology groups considered are continuous.

The symbol $F$ will denote either a number field or a finite extension of $\mathbb{Q}_p$ for some prime number $p$. Given any such $F$, we will define:
\begin{align*}
\cO_F&:=\mbox{the ring of integers in $F$};\\
F^c&:=\mbox{a fixed algebraic closure of $F$};\\
\cO_{F^c}&:=\mbox{the integral closure of $\cO_F$ in $F^c$};\\
\Omega_F&:=\Gal(F^c/F);\\
F^t&:=\mbox{the maximal tamely ramified extension of $F$ in $F^c$};\\
\Omega^t_F&:=\Gal(F^t/F);\\
M_F&:=\mbox{the set of all finite primes in $F$}.
\end{align*}
We will let $\Omega_F$ and $\Omega_F^t$ act trivially on $G$ on the left. We will also choose a compatible set $\{\zeta_n:n\in\mathbb{Z}^+\}$ of \mbox{primitive roots of} unity in $F^c$, that is, we have $(\zeta_{mn})^m=\zeta_n$ for all $m,n\in\mathbb{Z}^+$. For $G$ abelian, we will also write $\widehat{G}$ for the group of irreducible $F^c$-valued characters on $G$.

In the case that $F$ is a number field, for each $v\in M_F$ we will define:
\begin{align*}
F_v&:=\mbox{the completion of $F$ with respect to $v$};\\
i_v&:=\mbox{a fixed embedding $F^c\longrightarrow F_v^c$ extending the natural}\\
&\hspace{0.65cm}\mbox{embedding $F\longrightarrow F_v$};\\
\widetilde{i_v}&:=\mbox{the embedding $\Omega_{F_v}\longrightarrow\Omega_{F}$ induced by $i_v$}.
\end{align*}
By abuse of notation, we will also write $i_v$ for the isomorphism $F^c\longrightarrow i_v(F^c)$ induced by $i_v$ and $i_v^{-1}$ for the inverse of this isomorphism. Using this notation, the embedding $\widetilde{i_v}:\Omega_{F_v}\longrightarrow\Omega_F$ is then defined by
\begin{equation}\label{iv}
\widetilde{i_v}(\omega):=i_v^{-1}\circ\omega\circ i_v.
\end{equation}
Moreover, if $\{\zeta_n:n\in\mathbb{Z}^+\}$ is the chosen compatible set of primitive roots of unity in $F^c$, then for each $v\in M_F$, we will choose $\{i_v(\zeta_n):n\in\mathbb{Z}^+\}$ to be the compatible set of primitive roots of unity in $F^c_v$.

\subsection{Locally Free Class Groups}\label{classgroup} Let $F$ be number field. We \mbox{will recall the} definition and an idelic description of the locally free class group of $\cO_FG$ (see \cite[Chapter 6]{CR2} for more details).

\begin{definition}An \emph{$\cO_FG$-lattice} is a left $\cO_FG$-module which is finitely generated and projective as an $\cO_F$-module. Two $\cO_FG$-lattices $X$ and $X'$ \mbox{are said} to be \emph{stably isomorphic} if there exists $k\in\mathbb{Z}^+$ such that
\[
X\oplus(\cO_FG)^k\simeq X'\oplus(\cO_FG)^k.
\]
The stable isomorphism class of $X$ will be denoted by $[X]$.
\end{definition}

\begin{remark}\label{stable}If two $\cO_FG$-lattices are isomorphic, then they are certainly stably isomorphic. The converse holds as well when $G$ is abelian (see \cite[Proposition 51.2 and Theorem 51.24]{CR2}, for example).
\end{remark}

\begin{definition}
An $\cO_FG$-lattice $X$ is \emph{locally free over $\cO_FG$ (of rank one)} if $\cO_{F_v}\otimes_{\cO_F}X$ and $\cO_{F_v}G$ are isomorphic as $\cO_{F_v}G$-modules for all $v\in M_F$.
\end{definition}

\begin{definition}\label{LFCG}The \emph{locally free class group of $\cO_FG$} is defined to be the set
\[
\Cl(\cO_FG):=\{[X]:X\mbox{ is a locally free $\cO_FG$-lattice}\}
\]
equipped with the following group operation. By \cite[Corollary 31.7]{CR1}, for any pair of locally free $\cO_FG$-lattices $X$ and $X'$, there exists a locally free $\cO_FG$-lattice $X''$ such that $X\oplus X'\simeq\cO_FG\oplus X''$. It is straightforward to check that $[X'']$ is uniquely determined by $[X]$ and $[X']$. We then define $[X][X']:=[X'']$.
\end{definition}

We note that the group operation of $\Cl(\cO_FG)$ is usually written additively. We will write it multiplicatively instead because we will use an idelic description of $\Cl(\cO_FG)$, which we recall below.

\begin{definition}\label{J(FG)}Let $J(FG)$ denote the restricted direct product of the groups $(F_vG)^\times$ with respect to the subgroups $(\cO_{F_v}G)^\times$ for $v\in M_F$. Moreover, let
\[
\partial:(FG)^\times\longrightarrow J(FG)
\]
denote the diagonal map and let
\[
U(\cO_FG):=\prod_{v\in M_F}(\cO_{F_v}G)^\times
\]
be the group of unit ideles. 
\end{definition}

For each idele $c=(c_v)\in J(FG)$, define
\begin{equation}\label{OGc}
\cO_FG\cdot c:=\bigcap_{v\in M_F}(\cO_{F_v}G\cdot c_v\cap FG).
\end{equation}
Since every locally free $\cO_FG$-lattice may be embedded into $FG$, the map
\begin{equation}\label{j}
j:J(FG)\longrightarrow\mbox{Cl}(\cO_FG);\hspace{1em}j(c):=[\cO_FG\cdot c]
\end{equation}
is surjective. It is also a homomorphism by \cite[Theorem 31.19]{CR1}. 

\begin{thm}\label{isoLFCG}
If $G$ is abelian, then the map $j$ induces an isomorphism
\[
\mbox{Cl}(\cO_FG)\simeq\frac{J(FG)}{\partial((FG)^\times) U(\cO_FG)}.
\]
\end{thm}
\begin{proof}See \cite[Theorem 49.22 and Exercise 51.1]{CR2}, for example.
\end{proof}

\subsection{Galois Algebras and Resolvends}\label{galois}

Let $F$ be a number field or a finite extension of $\mathbb{Q}_p$. We will give a brief review of Galois algebras and resolvends (see \cite[Section 1]{M} for more details).

\begin{definition}\label{GaloisAlg}A \emph{$G$-Galois $F$-algebra} is a commutative semi-simple algebra $N$ over $F$ on which $G$ acts on the left as a group of automorphisms such that $N^{G}=F$ and $[N:F]=|G|$, where $[N:F]$ is the dimension of $N$ over $F$. Two $G$-Galois $F$-algebras are said to be \emph{isomorphic} if there is an $F$-algebra isomorphism between them which preserves the action of $G$.
\end{definition}

Consider the $F^c$-algebra $\Map(G,F^c)$ on which $G$ acts on the left by
\[
(s\cdot a)(t):=a(ts)\hspace{1cm}\mbox{for $a\in\Map(G,F^c)$ and $s,t\in G$}.
\]
Recall that $\Omega_F$ acts trivially on $G$ by definition. For each $h\in\Hom(\Omega_F,G)$, let $^hG$ denote the group $G$ endowed with the twisted $\Omega_F$-action given by
\[
\omega\cdot s:=h(\omega)s\hspace{1cm}\mbox{for $s\in G$ and $\omega\in\Omega_F$}.
\]
Now, consider the $F$-subalgebra and $G$-submodule
\[
F_{h}:=\Map_{\Omega_F}(^{h}G,F^{c})
\]
of $\Map(G,F^c)$ consisting of the maps $^hG\longrightarrow F^c$ that preserve the $\Omega_F$-action. If $\{s_i\}$ is a set of coset representatives of $h(\Omega_F)\backslash G$ and
\begin{equation}\label{Fh}
F^{h}:=(F^{c})^{\ker(h)},
\end{equation}
then evaluation at the elements $s_i$ induces an isomorphism
\[
F_{h}\simeq \prod_{h(\Omega_F)\backslash G}F^{h}
\]
of $F$-algebras. This implies that $[F_h:F]=[G:h(\Omega_F)][F^h:F]=|G|$. Notice that $(F_h)^G=F$ also, where $F$ is  identified with the set of constant $F$-valued functions in $F_h$. It follows that $F_h$ is a $G$-Galois $F$-algebra.

It is not difficult to check that every $G$-Galois $F$-algebra is \mbox{isomorphic to $F_h$} for some $h\in\Hom(\Omega_F,G)$, and that for $h,h'\in\Hom(\Omega_F,G)$ we have $F_{h}\simeq F_{h'}$ if and only if $h$ and $h'$ differ by an element in $\mbox{Inn}(G)$. Hence, the map $h\mapsto F_h$ induces a bijective correspondence between the pointed set
\[
H^1(\Omega_F,G):=\Hom(\Omega_F,G)/\mbox{Inn}(G)
\]
and the set of all isomorphism classes of $G$-Galois $F$-algebras.

In the case that $G$ is abelian, note that $H^1(\Omega_F,G)$ is equal to $\Hom(\Omega_F,G)$ and so in particular has a group structure.

\begin{definition}
Given $h\in\Hom(\Omega_F,G)$, define $F^h:=(F^c)^{\ker(h)}$ as in (\ref{Fh}). Let $\cO^h:=\cO_{F^h}$ and define the \emph{ring of integers of $F_h$} by
\[
\cO_h:=\mbox{Map}_{\Omega_F}(^hG,\cO^h).
\]
\end{definition}

\begin{remark}\label{localize}
Assume that $F$ a number field. Given $h\in\Hom(\Omega_F,G)$, define
\begin{equation}\label{hv}
h_v\in\Hom(\Omega_{F_v},G);\hspace{1em}h_v:=h\circ\widetilde{i_v}
\end{equation}
for each $v\in M_F$. It was proved in \cite[(1.4)]{M} that $(F_v)_{h_v}\simeq F_v\otimes_FF_h$. Similarly, we have $\cO_{h_v}\simeq\cO_{F_v}\otimes_{\cO_F}\cO_h$.
\end{remark}

\begin{definition}\label{ramification}
Given $h\in\Hom(\Omega_F,G)$, we say that $F_h/F$ or $h$ is \emph{unramified} (respectively, \emph{tame}) if $F^h/F$ is unramified (respectively, \emph{tame}).
\end{definition}

\begin{definition}\label{resolvend}
The \emph{resolvend map} $\mathbf{r}_{G}:\mbox{Map}(G,F^{c})\longrightarrow F^{c}G$ is defined by
\[
\mathbf{r}_{G}(a):=\sum\limits _{s\in G}a(s)s^{-1}.
\]
It is clear that $\mathbf{r}_{G}$ is an isomorphism of $F^cG$-modules, but not an isomorphism of $F^cG$-algebras because it does not preserve multiplication. 
\end{definition}

Given $a\in\mbox{Map}(G,F^c)$, it is easy to check that $a\in F_h$ if and only if
\begin{equation}\label{resol1}
\omega\cdot\mathbf{r}_{G}(a)=\mathbf{r}_{G}(a)h(\omega)
\hspace{1cm}\mbox{for all }\omega\in\Omega_F.
\end{equation}
The following proposition shows that resolvends may also be used to identify elements $a\in F_h$ for which $F_h=FG\cdot a$ or $\cO_h=\cO_FG\cdot a$.

\begin{prop}\label{NBG}Assume that $G$ is abelian and let $a\in F_h$.
\begin{enumerate}[(a)]
\item We have $F_h=FG\cdot a$ if and only if $\mathbf{r}_{G}(a)\in (F^{c}G)^{\times}$.
\item We have $\cO_h=\cO_FG\cdot a$ with $h$ unramified if and only if $\mathbf{r}_G(a)\in(\cO_{F^c}G)^\times$. Moreover, if $F$ is a finite extension of $\mathbb{Q}_p$ and $h$ is unramified, then there exists $a\in\cO_h$ such that $\cO_h=\cO_FG\cdot a$.
\end{enumerate}
\end{prop}
\begin{proof}See \cite[Proposition 1.8]{M} for (a) and \cite[(2.11)]{M} for the first claim in (b). For the second claim in (b), it follows from a classical theorem of Noether, or alternatively from \cite[Proposition 5.5]{M}. We remark that only the \mbox{first claim in} (b) requires the assumption that $G$ is abelian.
\end{proof}

\begin{remark}\label{tamesubgp}
A homomorphism $h\in\Hom(\Omega_F,G)$ is tame if and only if it factors through the quotient map $\Omega_F\longrightarrow\Omega_F^t$. Hence, the subset of $\Hom(\Omega_F,G)$ consisting of the tame homomorphisms may be identified with $\Hom(\Omega_F^t,G)$. In particular, a tame homomorphism $\widetilde{h}\in\Hom(\Omega_F,G)$ may be identified with the element $h\in\Hom(\Omega_F^t,G)$ defined by
\[
h(\omega):=\widetilde{h}(\widetilde{\omega})\hspace{1cm}\mbox{for }\omega\in\Omega_F^t,
\]
where $\widetilde{\omega}\in\Omega_F$ is any lift of $\omega$. Conversely, any element $h\in\Hom(\Omega_F^t,G)$ may be identified with the homomorphism $\widetilde{h}\in\Hom(\Omega_F,G)$ defined by
\[
\widetilde{h}(\omega):=h(\omega|_{F^t})\hspace{1cm}\mbox{for }\omega\in\Omega_F.
\]
The above identifications will be used repeatedly throughout the rest of this paper. In particular, given $h\in\Hom(\Omega_F^t,G)$, all of the definitions and results introduced in this subsection still apply.
\end{remark}

\section{The Basic Diagram}\label{basicdiagram}

Throughout this section, let $K/k$ denote a fixed Galois extension of number fields and set $\Sigma:=\Gal(K/k)$. We will also assume that $G$ is abelian and we fix a left $\Sigma$-module structure on $G$. The purpose of this section is to explain the construction of the basic diagram (\ref{bd}), which was first defined by Brinkhuis in \cite[Theorem 5.1]{B} in order to connect the study of realizable classes to that of embedding problems (cf. Remark~\ref{bdmodified} above and the discussion following it). The map $\gal$ in the basic diagram is that defined in (\ref{gal}), and the map
\[
i^*:H^2(\Sigma,G)\longrightarrow H^2(\Sigma,(\cO_KG)^\times)
\]
is that induced by the natural inclusion $G\longrightarrow(\cO_KG)^\times$. The horizontal rows will be constructed in the subsequent subsections. We will also show that the top row is exact, that all of the maps except possibly $\gal$ are homomorphisms, and that (\ref{bd}) commutes.

In this section, we will let $\Gal(K^t/k)$ act on $G$ on the left via the natural quotient map $\Gal(K^t/k)\longrightarrow\Sigma$ and the given left $\Sigma$-action on $G$. Via the natural left $\Gal(K^t/k)$-action on $K^t$, this extends to a left $\Gal(K^t/k)$-action on $K^tG$. As noted in Remark~\ref{tamesubgp}, we will \mbox{identify $\Hom(\Omega_K^t,G)$ with the} subset of $\Hom(\Omega_K,G)$ consisting of the tame homomorphisms. We will also use the following notation.

\begin{definition}\label{liftgamma}For each $\gamma\in\Sigma$, choose once and for all a lift $\lift{\gamma}\in\Gal(K^t/k)$ of $\gamma$ with $\lift{1}=1$
\end{definition}

\subsection{The Top Row: Hochschild-Serre Sequence}\label{top row}Recall that $\Omega_K^t$ acts trivially on $G$ on the left by definition. From the Hochschild-Serre spectral sequence (see \cite[Chapter I Section 2.6]{SG}, for example) associated to the group extension
\[
\begin{tikzcd}[column sep=1cm, row sep=1.5cm]
1 \arrow{r} &
\Omega_K^t \arrow{r} &
\Gal(K^t/k) \arrow{r} &
\Sigma\arrow{r} &
1,
\end{tikzcd}
\]
we then obtain an exact sequence
\begin{equation}\label{toprow}
\begin{tikzcd}[column sep=1.5cm]
H^1(\Gal(K^t/k),G)\arrow{r}{\mbox{res}} & \Hom(\Omega_K^t,G)^\Sigma\arrow{r}[font=\normalsize]{tr} & H^2(\Sigma,G).
\end{tikzcd}
\end{equation}
Here $\mbox{res}$ is given by restriction and $tr$ is the \emph{transgression map}. We remark that (\ref{toprow}) is also part of the five-term inflation-restriction exact sequence in group cohomology (see \cite[Proposition 1.6.6]{NG}, for example). We will recall the definitions of the $\Sigma$-action on $\Hom(\Omega_K^t,G)$ and the map $tr$ in this particular setting.

\begin{definition}\label{tr}Given  $h\in\Hom(\Omega_K^t,G)$ and $\gamma\in\Sigma$, define
\[
(h\cdot\gamma)(\omega):= \gamma^{-1}\cdot h(\lift{\gamma}\omega\lift{\gamma}^{-1})\hspace{1cm}\mbox{for all $\omega\in\Omega_K^t$}.
\]
This definition is independent of the choice of the lift $\lift{\gamma}$ because $G$ is abelian. The \emph{transgression map} $tr:\Hom(\Omega_K^t,G)^\Sigma\longrightarrow H^2(\Sigma,G)$ (see \cite[Proposition 1.6.5]{NG}, for example) is defined by

\[
tr(h):=[(\gamma,\delta)\mapsto h((\lift{\gamma})(\lift{\delta})(\lift{\gamma\delta})^{-1})],
\]
where $[-]$ denotes the cohomology class. This definition is also independent of the choice of the lifts $\lift{\gamma}$ for $\gamma\in\Sigma$.
\end{definition}

Below, we will explain how the exact sequence (\ref{toprow}) is related to the study of embedding problems. To that end, observe that each group extension
\[
\begin{tikzcd}[column sep=1cm]
E:\hspace{0.5cm}1\arrow{r}&G\arrow{r}[font=\normalsize]{\upiota}&\Gamma\arrow{r}&\Sigma\arrow{r}&1
\end{tikzcd}
\]
of $\Sigma$ by $G$ induces a canonical left $\Sigma$-module structure on $G$ via conjugation in $\Gamma$ as follows. For each $\gamma\in\Sigma$, choose a lift $\upsigma(\gamma)$ of $\gamma$ in $\Gamma$. Then, define
\begin{equation}\label{conjaction}
\gamma*s:=\upiota^{-1}(\upsigma(\gamma)\upiota(s)\upsigma(\gamma)^{-1})\hspace{1cm}
\mbox{for $\gamma\in\Gamma$ and $s\in G$}.
\end{equation}
This definition does not depend upon the choice of the lift $\upsigma(\gamma)$ because $G$ is abelian. In addition, define a map $c_{E}:\Sigma\times\Sigma\longrightarrow G$ by
\begin{equation}\label{cE}
c_{E}(\gamma,\delta):=\upiota^{-1}(\upsigma(\gamma)\upsigma(\delta)\upsigma(\gamma\delta)^{-1}).
\end{equation}
Let $E(K/k,G)$ denote the set of all equivalence classes of the group extensions of $\Sigma$ by $G$ for which (\ref{conjaction}) coincides with the $\Sigma$-action on $G$ that we have fixed. It is well-known (see \cite[Theorem 1.2.4]{NG}, for example) that $E\mapsto c_{E}$ induces a bijective correspondence between $E(K/k,G)$ and the  group $H^2(\Sigma,G)$, and the map $c_{E}$ represents the trivial cohomology class if and only if $E$ splits.

\begin{prop}\label{Brelation}Let $h\in\Hom(\Omega_K^t,G)^\Sigma$ be surjective. Then, \mbox{the field $K^h$ is} a tame solution to the embedding problem $(K/k,G,E_h)$, where $E_h$ is a group extension of $\Sigma$ by $G$ whose equivalence class corresponds to $tr(h)$.
\end{prop}
\begin{proof}Observe that $K^h/k$ is Galois because $\Gal(K^t/K^h)$, which equals $\ker(h)$, is a normal subgroup of $\Gal(K^t/k)$. To see why, let $\omega_k\in\Gal(K^t/k)$ be given and write $\omega_k=\lift{\gamma}\omega_0$, where $\gamma\in\Sigma$ and $\omega_0\in\Omega_K^t$. For any $\omega\in\ker(h)$, we have
\begin{align*}
h(\omega_k\omega\omega_k^{-1})
&=h(\lift{\gamma}\omega_0\omega\omega_0^{-1}\lift{\gamma}^{-1})\\
&=\gamma\cdot (h\cdot\gamma)(\omega_0\omega\omega_0^{-1})\\
&=\gamma\cdot (h(\omega_0)h(\omega)h(\omega_0)^{-1}),
\end{align*}
where the last equality holds because $h$ is $\Sigma$-invariant and is a homomorphism on $\Omega_K^t$. Since $\omega\in\ker(h)$, we see that $h(\omega_k\omega\omega_k^{-1})=1$ and so $\omega_k\omega\omega_k^{-1}\in\ker(h)$. \par\noindent  Hence, indeed $\ker(h)$ is normal in $\Gal(K^t/k)$.

Next, notice that $h$ induces an isomorphism $\lift{h}:\Gal(K^h/K)\longrightarrow G$ because $h$ is surjective. Set $\Gamma_h:=\Gal(K^h/k)$ and let $\upiota:G\longrightarrow\Gal(K^h/K)\longrightarrow\Gamma_h$ denote the homomorphism $\lift{h}^{-1}$ followed by the natural inclusion $\Gal(K^h/K)\longrightarrow\Gal(K^h/k)$. We then obtain a group extension
\[
\begin{tikzcd}[column sep=1cm]
E_h:\hspace{0.5cm}1\arrow{r}&G\arrow{r}[font=\normalsize]{\upiota}&\Gamma_h\arrow{r}&\Sigma\arrow{r}&1
\end{tikzcd}
\]
of $\Sigma$ by $G$. The diagram
\[
\begin{tikzpicture}[baseline=(current bounding box.center)]
\node at (-1,3) [name=A] {$1$};
\node at (13,3) [name=B] {$1$};
\node at (-1,1) [name=C] {$1$};
\node at (13,1) [name=D] {$1$};
\node at (2,3) [name=1] {$\Gal(K^h/K)$};
\node at (6,3) [name=2] {$\Gal(K^h/k)$};
\node at (10,3) [name=3] {$\Gal(K/k)$};
\node at (2,1) [name=4] {$G$};
\node at (6,1) [name=5] {$\Gamma_h$};
\node at (10,1) [name=6] {$\Sigma$};
\path[->]
(A) edge node[auto] {} (1)
(C) edge node[auto] {} (4)
(3) edge node[auto] {} (B)
(6) edge node[auto] {} (D)
(1) edge node[auto] {} (2)
(2) edge node[auto] {} (3)
(4) edge node[below] {$\upiota$} (5)
(5) edge node[auto] {} (6);
\path[->]
(1) edge node[auto] {$\lift{h}$} (4);
\draw[double equal sign distance]
(3) -- (6) (2) -- (5);
\end{tikzpicture}
\]
clearly commutes and $K^h/K$ is clearly tame. It follows that $K^h/K$ is a tame solution to the embedding problem $(K/k,G,E_h)$.

Finally, for each $\gamma\in\Sigma$, choose $\upsigma(\gamma):=\lift{\gamma}|_{K^h}$ to be a lift of $\gamma$ in $\Gamma_h$. Given an element $s\in G$, there exists $\omega\in\Omega_K^t$ such that $h(\omega)=s$ because $h$ is surjective. The left $\Sigma$-action on $G$ defined as in (\ref{conjaction}) is then given by
\begin{align*}
\gamma*s
&=\lift{h}((\lift{\gamma}|_{K^h})(\omega|_{K^h})(\lift{\gamma}|_{K^h})^{-1})\\
&=h(\lift{\gamma}\omega\lift{\gamma}^{-1})\\
&=\gamma\cdot s,
\end{align*}
where the last equality follows because $h$ is $\Sigma$-invariant. This shows that the equivalence class of $E_h$ lies in $E(K/k,G)$. Also, the map $c_{E_h}:\Sigma\times\Sigma\longrightarrow G$ in (\ref{cE}) is given by
\begin{align*}
c_{E_h}(\gamma,\delta)
&=\lift{h}((\lift{\gamma}|_{K^h})(\lift{\delta}|_{K^h})(\lift{\gamma\delta}|_{K^h})^{-1})\\
&=h((\lift{\gamma})(\lift{\delta})(\lift{\gamma\delta})^{-1})
\end{align*}
and so the equivalence class of $E_h$ corresponds to $tr(h)$. This completes the proof of the proposition.
\end{proof}

\subsection{The Bottom Row: Fr\"{o}hlich-Wall Sequence}Observe that $\cO_KG$ is equipped with a canonical left $\Sigma$-action, namely that induced by the given left $\Sigma$-action on $G$ and $\cO_K$. For all $\gamma\in\Sigma$ and $\beta,\beta'\in\cO_KG$, we have
\[
\gamma\cdot(\beta+\beta')=\gamma\cdot\beta+\gamma\cdot\beta'
\hspace{1cm}\mbox{and}\hspace{1cm}
\gamma\cdot(\beta\beta')=(\gamma\cdot\beta)(\gamma\cdot\beta').
\]
In other words, the ring $\cO_KG$ is a \emph{$\Sigma$-ring}.  From the Fr\"{o}hlich-Wall sequence associated to $\cO_KG$ (see \cite[Section 1]{B}, for example), we obtain a homomorphism
\[
\xi:\mbox{Cl}(\cO_KG)^\Sigma\longrightarrow H^2(\Sigma,(\cO_KG)^\times).
\]
We will recall the definitions of the $\Sigma$-action on $\mbox{Cl}(\cO_KG)$ and the map $\xi$.

\subsubsection{The left $\Sigma$-action on $\mbox{Cl}(\cO_KG)$}\label{s:6.2.1}

\begin{definition} Let $X$ and $X'$ be $\cO_KG$-modules. A \emph{semilinear isomorphism from $X$ to $X'$} is a group isomorphism $\varphi:X\longrightarrow X'$ satisfying
\[
\varphi(\beta\cdot x)=(\gamma\cdot\beta)\cdot\varphi(x)\hspace{1cm}\mbox{for all }\beta\in\cO_KG\mbox{ and }x\in X
\]
for some $\gamma\in\Sigma$. Any such $\gamma\in\Sigma$ is called a \emph{grading of $\varphi$}.
\end{definition}

\begin{definition}\label{Sigmaaction} Let $[X]\in\Cl(\cO_KG)$ and $\gamma\in\Sigma$. Define $\gamma\cdot[X]:=[Y]$ if there exists a semilinear isomorphism $\varphi:X\longrightarrow Y$ having $\gamma$ as a grading. Clearly the isomorphism class $[Y]$ of $Y$ (recall Remark~\ref{stable}) is uniquely determined by that of $X$. Note also that such a $Y$ always exists, as we may take $Y:=X_\gamma$ to be the abelian group $X$ equipped with the structure
\begin{equation}\label{Xgamma}
\beta* x:=(\gamma^{-1}\cdot\beta)\cdot x
\hspace{1cm}\mbox{for }\beta\in(\cO_KG)^\times\mbox{ and }x\in X_\gamma
\end{equation}
as an $\cO_KG$-module and take $\varphi=\mbox{id}_X$ to be the identity on $X$.
\end{definition}

It is clear that Definition~\ref{Sigmaaction} defines a left $\Sigma$-action on the group $\Cl(\cO_KG)$. The next proposition shows that $\Cl(\cO_KG)$ is in fact a left $\Sigma$-module under this action and so $\mbox{Cl}(\cO_KG)^\Sigma$ is a subgroup of $\mbox{Cl}(\cO_KG)$. 

\begin{prop}Let $[X],[X']\in\Cl(\cO_KG)$. For all $\gamma\in\Sigma$, we have
\[
\gamma\cdot([X][X'])=(\gamma\cdot[X])(\gamma\cdot[X']).
\]
\end{prop}
\begin{proof}Let $[X'']\in\Cl(\cO_KG)$ be such that $[X'']=[X][X']$. By Definition~\ref{LFCG}, this means that there exists an $\cO_KG$-isomorphism
\[
\varphi:X\oplus X'\longrightarrow\cO_KG\oplus X''.
\]
Let $X_\gamma$ denote the group $X$ equipped with the $\cO_KG$-structure given in (\ref{Xgamma}), and similarly for $X_\gamma'$ and $X_\gamma''$. Let $\varphi_\gamma:\cO_KG\longrightarrow\cO_KG$ denote the bijective map defined by $\beta\mapsto\gamma\cdot\beta$. Then, the map
\[
(\varphi_\gamma\oplus\mbox{id}_{X''})\circ\varphi:X_\gamma\oplus X'_\gamma\longrightarrow\cO_KG\oplus X_\gamma''
\]
is an isomorphism of $\cO_KG$-modules and so $[X_\gamma'']=[X_\gamma][X_\gamma']$, as desired.
\end{proof}

The next proposition ensures that diagram (\ref{bd}) is well-defined.

\begin{prop}\label{Im(galA)}Let $h\in\Hom(\Omega_K^t,G)^\Sigma$. For all $\gamma\in\Sigma$, the map
\begin{equation}\label{phigamma}
\varphi_\gamma:\mathbf{r}_G(\cO_h)\longrightarrow\mathbf{r}_G(\cO_h);
\hspace{1em}\varphi_\gamma(\mathbf{r}_G(a)):=\lift{\gamma}\cdot\mathbf{r}_G(a)
\end{equation}
is well-defined and is a semilinear isomorphism having $\gamma$ as a grading. Consequently, we have $\gamma\cdot[\cO_h]=[\cO_h]$ and so $\gal(\Hom(\Omega_K^t,G)^\Sigma)\subset\Cl(\cO_KG)^\Sigma$.
\end{prop}
\begin{proof}Let $\gamma\in\Sigma$ be given. First, we will verify \mbox{that $\varphi_\gamma(\mathbf{r}_G(\cO_h))\subset\mathbf{r}_G(\cO_h)$  so} that $\varphi_\gamma$ is well-defined. To that end, let $a\in\cO_h$ be given. Since $\mathbf{r}_G$ is bijective, there exists $a'\in\Map(G,K^c)$ such that $\mathbf{r}_G(a')=\lift{\gamma}\cdot\mathbf{r}_G(a)$. We will use (\ref{resol1}) to show that $a\in K_h$. Given $\omega\in\Omega_K^t$, note that $\lift{\gamma}^{-1}\omega\lift{\gamma}\in\Omega_K^t$ and that
\[
\lift{\gamma}^{-1}\omega\lift{\gamma}\cdot\mathbf{r}_G(a)=\mathbf{r}_G(a)h(\lift{\gamma}^{-1}\omega\lift{\gamma})
\]
since $a\in K_h$. This implies that
\begin{align*}
\omega\cdot\mathbf{r}_G(a')
&=\lift{\gamma}\cdot(\lift{\gamma}^{-1}\omega\lift{\gamma}\cdot\mathbf{r}_G(a))\\
&=\mathbf{r}_G(a')(h\cdot\gamma^{-1})(\omega)\\
&=\mathbf{r}_G(a')h(\omega),
\end{align*}
where the last equality follows because $h$ is $\Sigma$-invariant. It then follows from (\ref{resol1}) that $a'\in K_h$. Since $a\in\cO_h$, it is clear that $a'\in\cO_h$ as well. Hence, the map $\varphi_\gamma$ is well-defined. Once we have established that $\varphi_\gamma$ is well-defined, it is clear that $\varphi_\gamma$ is a semilinear isomorphism having $\gamma$ as a grading. Since $\mathbf{r}_G$ restricts to an $\cO_KG$-isomorphism $\cO_h\simeq\mathbf{r}_G(\cO_h)$, we \mbox{have $[\cO_h]=[\mathbf{r}_G(\cO_h)]$ and} the above shows that $\gamma\cdot[\cO_h]=[\cO_h]$.
\end{proof}

\subsubsection{The homomorphism $\xi$}

\begin{definition}Let $X$ be $\cO_KG$-module. Define $\mbox{Sem}(X)$ to be the set of all pairs $(\varphi,\gamma)$, where $\varphi:X\longrightarrow X$ is a semilinear isomorphism having $\gamma$ as a grading, equipped with the group operation $(\varphi,\gamma)(\varphi',\gamma'):=(\varphi\varphi',\gamma\gamma')$. Also, let $\mbox{Aut}(X)$ denote the group of $\cO_KG$-automorphisms on $X$. The map
\[
\mathfrak{g}_X:\mbox{Sem}(X)\longrightarrow\Sigma;\hspace{1em}\mathfrak{g}_X(\varphi,\gamma):=\gamma
\]
is then a homomorphism with $\ker(\mathfrak{g}_X)=\mbox{Aut}(X)$.
\end{definition}

Now, consider an element $[X]\in\mbox{Cl}(\cO_KG)^\Sigma$.   The map $\mathfrak{g}_X$ is then surjective because $[X]$ is $\Sigma$-invariant. Since $X$ is locally free over $\cO_KG$ (of rank one), an $\cO_KG$-automorphism on $X$ is of the form $\psi_\beta:x\mapsto\beta\cdot x$ for some $\beta\in(\cO_KG)^\times$. So, we may identify $\mbox{Aut}(X)$ with $(\cO_KG)^\times$. We then obtain a group extension
\[
\begin{tikzcd}[column sep=1cm]
E_X:\hspace{0.5cm}1 \arrow{r} &
(\cO_KG)^\times\arrow{r}[font=\normalsize, auto]{\mathfrak{i}_X} &
\mbox{Sem}(X) \arrow{r}[font=\normalsize, auto]{\mathfrak{g}_X} &
\Sigma \arrow{r} &
1
\end{tikzcd}
\]
of $\Sigma$ by $(\cO_KG)^\times$, where $\mathfrak{i}_X(\beta):=(\psi_\beta,1)$. Observe that $E_X$ induces a left $\Sigma$- module structure on $(\cO_KG)^\times$ via conjugation in $\mbox{Sem}(X)$ as follows. For each $\gamma\in\Sigma$, choose a lift $(\varphi_\gamma,\gamma)$ of $\gamma$ in $\mbox{Sem}(\Sigma)$. Then, define (cf. (\ref{conjaction}))
\begin{equation}\label{conjaction'}
\gamma*\beta:=\upiota_X^{-1}((\varphi_\gamma\psi_\beta\varphi_\gamma^{-1},1))
\hspace{1cm}\mbox{for $\gamma\in\Gamma$ and $\beta\in(\cO_KG)^\times$}.
\end{equation}
But for any $x\in(\cO_KG)^\times$, we have 
\[
(\varphi_\gamma\psi_\beta\varphi_\gamma^{-1})(x)=\varphi_\gamma(\beta\cdot\varphi_\gamma^{-1}(x))=(\gamma\cdot\beta)\cdot x=\varphi_{\gamma\cdot\beta}(x).
\]
This means that (\ref{conjaction'}) agrees with the existing left $\Sigma$-action on $(\cO_KG)^\times$. So, analogously to the bijective correspondence between $E(K/k,G)$ and $H^2(\Sigma,G)$ described in Subsection~\ref{top row}, the group extension $E_X$ corresponds to a cohomology class in $H^2(\Sigma,(\cO_KG)^\times)$. In particular, the class is represented by the $2$-cocycle $d_{X}:\Sigma\times\Sigma\longrightarrow(\cO_KG)^\times$ determined by the equations (cf. (\ref{cE}))
\begin{equation}\label{dX}
d_X(\gamma,\delta)\cdot x=(\varphi_\gamma\varphi_\delta\varphi_{\gamma\delta}^{-1})(x)\hspace{1cm}
\mbox{for all }x\in X.
\end{equation}

\begin{definition}\label{xi}Define $\xi:\Cl(\cO_KG)^\Sigma\longrightarrow H^2(\Sigma,(\cO_KG)^\times)$ by $\xi([X]):=[d_X]$, where $[-]$ denotes the cohomology class. It is not hard to check that this definition depends only on the isomorphism class $[X]$ of $X$ (recall Remark~\ref{stable}).
\end{definition}

\begin{prop}\label{xihom}The map $\xi$ is a homomorphism.
\end{prop}
\begin{proof}Given $[X],[X']\in\Cl(\cO_KG)^\Sigma$, define $X'':=X\otimes_{\cO_KG}X'$. Because $G$ is abelian, \cite[Theorem 55.16]{CR2} implies that $X''$ is locally free over $\cO_KG$ (of rank one), and we have $[X][X']=[X'']$ (cf. the proof of \cite[Theorem 55.26]{CR2}).

For each $\gamma\in\Sigma$, let $\varphi_\gamma$ and $\varphi_\gamma'$ be semilinear \mbox{automorphisms on $X$ and $X'$,} respectively, having $\gamma$ as a grading. Such automorphisms exist because both $[X]$ and $[X']$ are $\Sigma$-invariant. It is clear that $\varphi_{\gamma}'':=\varphi_\gamma\otimes\varphi_\gamma'$ is then a semilinear automorphism on $X''$ having $\gamma$ as a grading. Let $d_X,d_{X'},$ and $d_{X''}$ be defined as in (\ref{dX}). Then, for all $\gamma,\delta\in\Sigma$, $x\in X$, and $x'\in X'$, we have
\begin{align*}
d_{X''}(\gamma,\delta)\cdot (x\otimes x')
&=(\varphi_\gamma''\varphi_\delta''\varphi_{\gamma\delta}''^{-1})(x\otimes x')\\
&=(\varphi_\gamma\varphi_\delta\varphi_{\gamma\delta}^{-1})(x)\otimes(\varphi_\gamma'\varphi_\delta'\varphi_{\gamma\delta}'^{-1})(x')\\
&=(d_X(\gamma,\delta)\cdot x)\otimes(d_{X'}(\gamma,\delta)\cdot x')\\
&=(d_X(\gamma,\delta)d_{X'}(\gamma,\delta))\cdot(x\otimes x').
\end{align*}
This shows that $d_{X''}=d_Xd_{X'}$ and so $\xi([X''])=\xi([X])\xi([X'])$, as desired.
\end{proof}

\subsection{Commutativity} We now give a proof that the basic diagram (\ref{bd}) is commutative.

\begin{thm}\label{commutes}
The basic diagram (\ref{bd}) is commutative. Moreover, the row at the top is exact, and all of the maps except possibly gal are homomorphisms.
\end{thm}
\begin{proof}Notice that diagram (\ref{bd}) is well-defined by Proposition~\ref{Im(galA)}. Now, we already know that the top row is exact. The maps $\mbox{res}$, $tr$, and $i^*$ are plainly homomorphisms, and $\xi$ is a homomorphism by Proposition~\ref{xihom}. Hence, it remains to verify the equality $i^*\circ tr=\xi\circ\gal$.

So, let $h\in\Hom(\Omega_K^t,G)^\Sigma$ be given. By Definition~\ref{tr}, the class $(i^*\circ tr)(h)$ is represented by the $2$-cocycle $d:\Sigma\times\Sigma\longrightarrow(\cO_KG)^\times$ defined by
\[
d(\gamma,\delta):=h((\lift{\gamma})(\lift{\delta})(\lift{\gamma\delta})^{-1}).
\]
Next, set $X:=\mathbf{r}_G(\cO_h)$ and notice that $\mathbf{r}_G$ restricts to an $\cO_KG$-isomorphism $\cO_h\simeq X$, so $\gal(h)=[X]$. For each $\gamma\in\Sigma$, let $\varphi_\gamma:X\longrightarrow X$ \mbox{be defined as} in (\ref{phigamma}), which is a semilinear isomorphism having $\gamma$ as a grading by Proposition~\ref{Im(galA)}. By Definition~\ref{xi}, the class $(\xi\circ\gal)(h)$ is then represented by the $2$-cocycle $d_X:\Sigma\times\Sigma\longrightarrow(\cO_KG)^\times$ defined by the equations
\[
d_X(\gamma,\delta)\cdot x=((\lift{\gamma})(\lift{\delta})(\lift{\gamma\delta})^{-1})\cdot x\hspace{1cm}\mbox{for all }x\in X.
\]
But $(\lift{\gamma})(\lift{\delta})(\lift{\gamma\delta})^{-1}\in\Omega_K^t$ and $x\in\mathbf{r}_G(\cO_h)$. It then follows from (\ref{resol1}) that
\[
((\lift{\gamma})(\lift{\delta})(\lift{\gamma\delta})^{-1})\cdot x=h((\lift{\gamma})(\lift{\delta})(\lift{\gamma\delta})^{-1}))\cdot x
\hspace{1cm}\mbox{for all }x\in X.
\]
This shows that $d_X=d$, whence $(i^*\circ tr)(h)=(\xi\circ\mbox{gal})(h)$, as desired.
\end{proof}

\section{Characterization of Realizable Classes}\label{R}

Throughout this section, assume that $G$ is abelian. For the moment, let $F$ be a number field. Given a tame $h\in\Hom(\Omega_F,G)$, recall that \mbox{a classical theo-} rem of Noether implies that $\cO_h$ is locally free over $\cO_FG$, and so $\cO_h$ defines a class $\cl(\cO_h)$ in $\Cl(\cO_FG)$. 

Below, we will explain how the class $\cl(\cO_h)$ may be computed \mbox{using resol-} vends. It will be helpful to recall the notation introduced in (\ref{OGc}) and (\ref{j}). First, since $\cO_h$ is locally free over $\cO_FG$ (of rank one), for each $v\in M_F$ there exists $a_v\in\cO_{h_v}$ such that
\vspace{-1mm}
\begin{equation}\label{av}
\cO_{h_v}=\cO_{F_v}G\cdot a_v.
\end{equation}
Next, by the Normal Basis Theorem, there exists $b\in F_h$ such that
\begin{equation}\label{b}
F_h=FG\cdot b.
\end{equation}
Since $F_vG\cdot a_v=F_{h_v}=F_vG\cdot b$ for all $v\in M_F$ and $\cO_{F_v}G\cdot a_v=\cO_{F_v}G\cdot b$ for all but finitely may $v\in M_F$, there exists $c=(c_v)\in J(FG)$ such that
\begin{equation}\label{cv}
a_v=c_v\cdot b
\end{equation}
for all $v\in M_F$. The $FG$-module isomorphism $FG\longrightarrow F_h$ defined by $\beta\mapsto\beta\cdot b$ then restricts to an $\cO_FG$-module isomorphism $\cO_FG\cdot c\longrightarrow \cO_h$. It follows that $\cl(\cO_h)=j(c)$. Now, recall that the resolvend map $\mathbf{r}_G:\mbox{Map}(G,F_v^c)\longrightarrow F_v^cG$ is an $F_vG$-module isomorphism. Equation (\ref{cv}) is then equivalent to
\begin{equation}\label{eq1}
\mathbf{r}_G(a_v)=c_v\cdot\mathbf{r}_G(b).
\end{equation}
This means that in order to compute the class $\cl(\cO_h)$, it suffices to compute  \par\noindent  the resolvends $\mathbf{r}_G(b)$ and $\mathbf{r}_G(a_v)$. Note that the resolvend $\mathbf{r}_G(b)$ of an element $b\in F_h$ satisfying (\ref{b}) is already characterized by Proposition~\ref{NBG} (a). 

The purpose of this section is to recall the characterization of $\cl(\cO_h)$ proved by McCulloh in \cite{M} (see Theorem~\ref{char1} below and the discussion following (\ref{char2})). To avoid repetition, we will only give an overview of the main ideas involved below, and then recall the necessary definitions in the subsequent subsections. 

As we already saw above, the characterization of $\cl(\cO_h)$ reduces to that of the resolvend $\mathbf{r}_G(a_v)$ of an element $a_v\in\cO_{h_v}$ satisfying (\ref{av}) for each $v\in M_F$. \par\noindent To that end, recall from Remark~\ref{tamesubgp} that $h_v$ may be regarded as an element of $\Hom(\Omega_{F_v}^t,G)$. We will see in Definition~\ref{factorh} that $h_v$ factors into $h_v=h_v^{nr}h_{v}^{tot}$, where $h_{v}^{nr},h_v^{tot}\in\Hom(\Omega_{F_v}^t,G)$
are such that $h_v^{nr}$ is unramified and $F_v^{h_v^{tot}}/F_v$ is totally ramified. As in the proof of \cite[Theorem 5.6]{M}, we may then decompose the resolvend $\mathbf{r}_G(a_v)$ as
\vspace{-1mm}
\begin{equation}\label{eq2}
\mathbf{r}_G(a_v)=\mathbf{r}_G(a_{v,nr})\mathbf{r}_G(a_{v,tot}),
\end{equation}
where $\cO_{h_v^{nr}}=\cO_{F_v}G\cdot a_{v,nr}$ and $\cO_{h_v^{tot}}=\cO_{F_v}G\cdot a_{v,tot}$. Notice that the resolvend $\mathbf{r}_G(a_{v,nr})$ of such an element $a_{v,nr}\in\cO_{h_v^{nr}}$ is already characterized by Propo- sition~\ref{NBG} (b). As for the resolvend $\mathbf{r}_G(a_{v,tot})$, it may be characterized using the \emph{Stickelberger transpose} and \emph{local prime $\mathfrak{F}$-elements}, which we will define in Subsections~\ref{Stickel sec} and~\ref{local F sec}, respectively. Moreover, rather than resolvends we will in fact use \emph{reduced resolvends}, which we will define in Subsection~\ref{reduced}.

\subsection{Cohomology and Reduced Resolvends}\label{reduced}

Let $F$ be a number field or a finite extension of $\mathbb{Q}_p$. Following \cite[Sections 1 and 2]{M}, \mbox{we will define reduced} resolvends and then interpret them as functions on characters of $G$.

Recall that $\Omega_F$ acts trivially on $G$ by definition. Define
\[
\mathcal{H}(FG):=((F^cG)^\times/G)^{\Omega_F}
\hspace{1cm}\mbox{and}\hspace{1cm}
\mathcal{H}(\cO_FG):=((\cO_{F^c}G)^\times/G)^{\Omega_F}.
\]
Taking $\Omega_F$-cohomology of the short exact sequence
\begin{equation}\label{exact1}
\begin{tikzcd}[column sep=1cm, row sep=1.5cm]
1 \arrow{r} &
G \arrow{r} &
(F^{c}G)^{\times} \arrow{r} &
(F^{c}G)^{\times}/G \arrow{r}&
1
\end{tikzcd}
\end{equation}
then yields the exact sequence
\begin{equation}\label{exact rag}
\begin{tikzcd}[column sep=1cm, row sep=1.5cm]
1 \arrow{r} &
G \arrow{r} &
(FG)^{\times} \arrow{r}[font=\normalsize]{rag} &
\mathcal{H}(FG) \arrow{r}[font=\normalsize]{\delta}&
\Hom(\Omega_F,G) \arrow{r}&
1,
\end{tikzcd}
\end{equation}
where $H^1(\Omega_F,(F^cG)^\times)=1$ follows from Hilbert's Theorem 90. Alternatively, given $h\in\Hom(\Omega_F,G)$, observe that a coset $\mathbf{r}_G(a)G\in\mathcal{H}(FG)$ belongs to the preimage of $h$ under $\delta$ if and only if
\[
h(\omega)=\mathbf{r}_G(a)^{-1}(\omega\cdot\mathbf{r}_G(a))\hspace{1cm}
\mbox{for all }\omega\in\Omega_F,
\]
which is equivalent to $F_h=FG\cdot a$ by (\ref{resol1}) and Proposition~\ref{NBG} (a). By the Normal Basis Theorem, such an element $a\in F_h$ always exists. So, indeed the map $\delta$ is surjective.

The same argument above also shows that
\begin{equation}\label{global}
\mathcal{H}(FG)=\{\mathbf{r}_G(a)G\mid F_h=FG\cdot a\mbox{ for some }h\in\Hom(\Omega_F,G)\}.
\end{equation}
Similarly, the argument above together with Proposition~\ref{NBG} (b) imply that
\[
\mathcal{H}(\cO_FG)=\{\mathbf{r}_G(a)G\mid\cO_h=\cO_FG\cdot a\mbox{ for some unramified }h\in\Hom(\Omega_F,G)\}.
\]

\begin{definition}\label{ha}Given $\mathbf{r}_G(a)G\in\mathcal{H}(FG)$, define $r_G(a):=\mathbf{r}_G(a)G$, called the \emph{reduced resolvend of $a$}. Moreover, define $h_a\in\Hom(\Omega_F,G)$ by
\[
h_a(\omega):=\mathbf{r}_G(a)^{-1}(\omega\cdot\mathbf{r}_G(a)),
\]
called the \emph{homomorphism associated to $r_G(a)$}. This definition clearly does not depend upon the choice of the representative $\mathbf{r}_G(a)$, and we have $F_h=FG\cdot a$ by (\ref{resol1}) and Proposition~\ref{NBG} (a).
\end{definition}

\begin{definition}\label{rag}Assume that $F$ is a number field. Let $J(\mathcal{H}(FG))$ denote the restricted direct product of the groups $\mathcal{H}(F_vG)$ with respect to the subgroups $\mathcal{H}(\cO_{F_v}G)$ for $v\in M_F$. Moreover, let
\[
\eta:\mathcal{H}(FG)\longrightarrow J(\mathcal{H}(FG))
\]
denote the diagonal map and let
\[
U(\mathcal{H}(\cO_FG)):=\prod_{v\in M_F}\mathcal{H}(\cO_{F_v}G)
\]
be the group of unit ideles. 

Next, recall Definition~\ref{J(FG)} and notice that the homomorphism
\begin{equation}\label{rag1}
\prod_{v\in M_F}rag_{F_v}:J(FG)\longrightarrow J(\mathcal{H}(FG))
\end{equation}
is clearly well-defined, where $rag_{F_v}$ is the map in (\ref{exact rag}). The diagram
\[
\begin{tikzpicture}[baseline=(current bounding box.center)]
\node at (0,2.5) [name=13] {$(FG)^\times$};
\node at (5.5,2.5) [name=14] {$J(FG)$};
\node at (0,0) [name=23] {$\mathcal{H}(FG)$};
\node at (5.5,0) [name=24] {$J(\mathcal{H}(FG))$};
\path[->,font=\large]
(13) edge node[auto]{$\partial$} (14)
(23) edge node[below]{$\eta$} (24)
(14) edge node[right]{$\displaystyle\prod_v rag_{F_v}$} (24)
(13) edge node[left]{$rag_F$} (23);
\end{tikzpicture}
\]
clearly commutes. By abuse of notation, we will also denote the map in (\ref{rag1}) by $rag=rag_F$.
\end{definition}

To interpret reduced resolvends as functions on characters of $G$, first recall that $\widehat{G}$ denotes the group of irreducible $F^c$-valued characters on $G$. Define
\[
\det:\mathbb{Z}\widehat{G}\longrightarrow\widehat{G};
\hspace{1em}\det\left(\sum_\chi n_\chi\chi\right):=\prod_\chi\chi^{n_\chi}
\]
and set\vspace{-1mm}
\begin{equation}\label{SG}
A_{\widehat{G}}:=\ker(\det).
\end{equation}
Applying the functor $\Hom(-,(F^c)^\times)$ to the short exact sequence
\[
\begin{tikzcd}[column sep=1.3cm, row sep=1.5cm]
1 \arrow{r} &
A_{\widehat{G}} \arrow{r} &
\mathbb{Z}\widehat{G}\arrow{r}[font=\normalsize]{\det} &
\widehat{G} \arrow{r}&
1
\end{tikzcd}
\]
then yields the short exact sequence
\begin{equation}\label{exact2}
\begin{tikzcd}[column sep=0.35cm, row sep=1.5cm]
1 \arrow{r} &
\Hom(\widehat{G},(F^{c})^{\times}) \arrow{r} &
\Hom(\mathbb{Z}\widehat{G},(F^{c})^{\times}) \arrow{r}&
\Hom(A_{\widehat{G}},(F^{c})^{\times}) \arrow{r}&
1,
\end{tikzcd}
\end{equation}
where exactness on the right follows from the fact that $(F^c)^\times$ is divisible and hence injective. We will identify (\ref{exact1}) with  (\ref{exact2}) as follows.

First, we have canonical identifications
\[
(F^{c}G)^{\times}=\mbox{Map}(\widehat{G},(F^{c})^{\times})
=\Hom(\mathbb{Z}\widehat{G},(F^c)^\times),
\]
where the second identification is given by extending the maps $\widehat{G}\longrightarrow (F^c)^\times$ via $\mathbb{Z}$-linearity, and the first is induced by characters as follows. Each resolvend $\mathbf{r}_{G}(a)\in (F^cG)^\times$ gives rise to a map $\varphi\in\mbox{Map}(\widehat{G},(F^c)^\times)$ defined by
\begin{equation}\label{resolvent}
\varphi(\chi):=\sum_{s\in G}a(s)\chi(s)^{-1}\hspace{1cm}\mbox{for }\chi\in\widehat{G}.
\end{equation}
Conversely, given $\varphi\in\mbox{Map}(\widehat{G},(F^c)^\times)$, one recovers  $\mathbf{r}_{G}(a)$ by the formula
\begin{equation}\label{Fourier}
a(s):=\frac{1}{|G|}\sum_{\chi}\varphi(\chi)\chi(s)\hspace{1cm}\mbox{for }s\in G.
\end{equation}
Since $G=\Hom(\widehat{G},(F^c)^\times)$ canonically, the third terms
\[
(F^cG)^\times/G=\Hom(A_{\widehat{G}},(F^c)^\times)
\]
in (\ref{exact1}) and (\ref{exact2}), respectively, are naturally identified as well.

Taking $\Omega_F$-invariants, we then obtain the identification
\begin{equation}\label{iden}
\mathcal{H}(FG)=\Hom_{\Omega_F}(A_{\widehat{G}},(F^c)^\times).
\end{equation}
Under this identification, it is clear from (\ref{resolvent}) that
\begin{equation}\label{integral}
\mathcal{H}(\cO_FG)\subset\Hom_{\Omega_F}(A_{\widehat{G}},\cO_{F^c}^\times).
\end{equation}
From (\ref{Fourier}), we easily see that the above inclusion is an equality when $F$ is a finite extension of $\mathbb{Q}_p$ for a prime $p$ not dividing $|G|$.

\subsection{The Stickelberger Transpose}\label{Stickel sec}

Let $F$ be a number field \mbox{or a finite ex-} tension of $\mathbb{Q}_p$ and let $\{\zeta_n:n\in\mathbb{Z}^+\}$ be the chosen compatible set of primitive roots of unity in $F^c$. We will recall the definition of the so-called Stickelberger transpose, which was first introduced by McCulloh in \cite[Section 4]{M} (see \cite[Proposition 5.4]{M} for the motivation of the definition).

\begin{definition}\label{Stickel}For each $\chi\in\widehat{G}$ and $s\in G$, let $\upsilon(\chi,s)\in\{0,1,\dots,|s|-1\}$ be the unique integer such that $\chi(s)=(\zeta_{|s|})^{\upsilon(\chi,s)}$ and define
\[
\langle\chi,s\rangle:=\upsilon(\chi,s)/|s|.
\]
Extending this definition by $\mathbb{Q}$-linearity, we obtain a pairing
\[
\langle\hspace{1mm},\hspace{1mm}\rangle:\mathbb{Q}\widehat{G}\times\mathbb{Q}G\longrightarrow\mathbb{Q},
\]
called the \emph{Stickelberger pairing}. The map
\[
\Theta:\mathbb{Q}\widehat{G}\longrightarrow\mathbb{Q}G;
\hspace{1em}
\Theta(\psi):=\sum_{s\in G}\langle\psi,s\rangle s
\]
is called the \emph{Stickelberger map}.
\end{definition}

\begin{prop}\label{A-ZG}For $\psi\in\mathbb{Z}\widehat{G}$, we have $\Theta(\psi)\in\mathbb{Z}G$ if and only if $\psi\in A_{\widehat{G}}$.
\end{prop}
\begin{proof}See \cite[Proposition 4.3]{M}.
\end{proof}

Up until now, we have let $\Omega_F$ act trivially on $G$. Below, we introduce other $\Omega_F$-actions on $G$, one of which will make the $\mathbb{Q}$-linear map $\Theta:\mathbb{Q}\widehat{G}\longrightarrow\mathbb{Q}G$ preserve the $\Omega_F$-action. Here, the $\Omega_F$-action on $\widehat{G}$ is the canonical one induced by the $\Omega_F$-action on the roots of unity in $F^c$.

\begin{definition}\label{cyclotomic}
Let $m:=\exp(G)$ and let $\mu_m$ be the group of $m$-th roots of unity in $F^c$. The \emph{$m$-th cyclotomic character of $\Omega_F$} is the homomorphism
\[
\kappa:\Omega_F\longrightarrow(\mathbb{Z}/m\mathbb{Z})^{\times}
\]
defined by the equations
\[
\omega(\zeta)=\zeta^{\kappa(\omega)}\hspace{1cm}\mbox{for $\omega\in\Omega_F$ and }\zeta\in\mu_m.
\]
For $n\in\mathbb{Z}$, let $G(n)$ be the group $G$ equipped with the $\Omega_F$-action given by
\[
\omega\cdot s:=s^{\kappa(\omega^{n})}\hspace{1cm}\mbox{for $s\in G$ and $\omega\in\Omega_F$}.
\]
\end{definition}

We will need $G(-1)$. Of course, if $F$ contains all $\exp(G)$-th roots of unity, then $\kappa$ is trivial and $\Omega_F$ acts trivially on $G(n)=G(0)$ for all $n\in\mathbb{Z}$.

\begin{prop}\label{eqvariant}
The map $\Theta:\mathbb{Q}\widehat{G}\longrightarrow\mathbb{Q}G(-1)$ preserves the $\Omega_F$-action.
\end{prop}
\begin{proof}See \cite[Proposition 4.5]{M}.
\end{proof}

From Propositions~\ref{A-ZG} and \ref{eqvariant}, we obtain an $\Omega_F$-equivariant map
\[
\Theta:A_{\widehat{G}}\longrightarrow\mathbb{Z}G(-1),
\]
which in turn yields an $\Omega_F$-equivariant homomorphism
\[
\Theta^{t}:\Hom(\mathbb{Z}G(-1),(F^{c})^{\times})\longrightarrow\Hom(A_{\widehat{G}},(F^{c})^{\times});\hspace{1em}f\mapsto f\circ\Theta.
\]
Via restriction, we then obtain a homomorphism
\[
\Theta^t=\Theta^t_{F}:\Hom_{\Omega_F}(\mathbb{Z}G(-1),(F^{c})^{\times})\longrightarrow\Hom_{\Omega_F}(A_{\widehat{G}},(F^{c})^{\times}),
\]
called the \emph{modified Stickelberger transpose}. 

Notice that we have a natural identification
\[
\Hom_{\Omega_F}(\mathbb{Z}G(-1),(F^c)^\times)
=\mbox{Map}_{\Omega_F}(G(-1),(F^c)^\times).
\]
To simplify notation, let
\begin{align}
\label{lambda1}\Lambda(FG)&:=\mbox{Map}_{\Omega_F}(G(-1),F^c);\\\notag
\Lambda(\cO_FG)&:=\mbox{Map}_{\Omega_F}(G(-1),\cO_{F^c}).
\end{align}
Then, we may regard $\Theta^{t}$ as a map 
\[
\Theta^t:\Lambda(FG)^{\times}\longrightarrow\mathcal{H}(FG)
\]
(recall the identification (\ref{iden})).

\begin{definition}\label{Theta}Assume that $F$ is a number field. Let $J(\Lambda(FG))$ denote the restricted direct product of the groups $\Lambda(F_vG)^\times$ with respect to the subgroups $\Lambda(\cO_{F_v}G)^\times$ for $v\in M_F$. Moreover, let
\[
\lambda=\lambda_F:\Lambda(FG)^\times\longrightarrow J(\Lambda(FG))
\]
denote the diagonal map and let
\[
U(\Lambda(\cO_FG)):=\prod_{v\in M_F}\Lambda(\cO_{F_v}G)^\times
\]
be the group of unit ideles.

Next, recall Definition~\ref{rag} and observe that the homomorphism
\begin{equation}\label{Theta1}
\prod_{v\in M_F}\Theta^t_{F_v}:J(\Lambda(FG))\longrightarrow J(\mathcal{H}(FG))
\end{equation}
is well-defined because the inclusion in (\ref{integral}) is an equality for all but finitely many $v\in M_F$. Recall from Section~\ref{notation} that we chose $\{i_v(\zeta_n):n\in\mathbb{Z}^+\}$ to be the compatible set of primitive roots of unity in $F_v^c$. Hence, the diagram
\begin{equation}\label{Theta'}
\begin{tikzpicture}[baseline=(current bounding box.center)]
\node at (0,2.5) [name=13] {$\Lambda(FG)^\times$};
\node at (5.5,2.5) [name=14] {$J(\Lambda(FG))$};
\node at (0,0) [name=23] {$\mathcal{H}(FG)$};
\node at (5.5,0) [name=24] {$J(\mathcal{H}(FG))$};
\path[->, font=\large]
(13) edge node[auto]{$\lambda$} (14)
(23) edge node[below]{$\eta$} (24)
(14) edge node[right]{$\displaystyle\prod_v\Theta^t_{F_v}$} (24)
(13) edge node[left]{$\Theta^t_{F}$} (23);
\end{tikzpicture}\vspace{-0.5mm}
\end{equation}
commutes. By abuse of notation, we will also use $\Theta^t=\Theta^t_{F}$ to denote the map in (\ref{Theta1}).
\end{definition}

\subsection{Local Prime $\mathfrak{F}$-Elements}\label{local F sec} Let $F$ be a finite extension of $\mathbb{Q}_p$. In what follows, let  $\pi=\pi_F$ denote a chosen uniformizer in $F$ and write $q=q_F$ for the order of the residue field $\cO_F/(\pi_F)$. Let $\{\zeta_n:n\in\mathbb{Z}^+\}$ denote the chosen compatible set of primitive roots of unity in $F^c$. Also, let $F^{nr}$ denote the maximal unramified extension of $F$ contained in $F^c$ and set $\Omega_F^{nr}:=\Gal(F^{nr}/F)$.

Recall from Remark~\ref{tamesubgp} that a tame $h\in\Hom(\Omega_F,G)$ may be regarded as an element in $\Hom(\Omega_F^t,G)$. Such a homomorphism admits a factorization as follows. First, we will recall the structure of the group $\Omega_F^t$ (see \cite[Sections 7 and 8]{F-CF}, for example). On one hand, the field $F^{nr}$ is obtained by adjoining to $F$ all roots of unity of order coprime to $p$. This means that $\Omega_F^{nr}$ is a procyclic group topologically generated by the Frobenius automorphism $\phi_F=\phi$, which is given by
\[
\phi(\zeta_n)=\zeta_n^q\hspace{1cm}\mbox{for all }(n,p)=1.
\]
On the other hand, the field $F^t$ is obtained by adjoining to $F^{nr}$ the $n$-th roots of $\pi$ for all $n$ coprime to $p$. We will choose a coherent set of radicals 
\[
\{\pi^{1/n}:n\in\mathbb{Z}^+\}
\] 
such that $(\pi^{1/mn})^{n}=\pi^{1/m}$ and then define $\pi^{m/n}:=(\pi^{1/n})^{m}$ for all $m,n\in\mathbb{Z}^+$. These choices of radicals then determine a distinguished topological generator $\sigma=\sigma_F$ of the procyclic group $\Gal(F^{t}/F^{nr})$, which is given by
\begin{equation}\label{sigma}
\sigma(\pi^{1/n})=\zeta_n\pi^{1/n}\hspace{1cm}\mbox{for all }(n,p)=1.
\end{equation}
Letting $\phi$ also denote the lifting of $\phi$ in $\Omega_F^{t}$ that fixes the radicals $\pi^{1/n}$ for all \par\noindent $(n,p)=1$, we then see that $\Omega_F^{t}$ is topologically generated by $\phi$ and $\sigma$. Notice that we have the relation $\phi\sigma\phi^{-1}\sigma^{-1}=\sigma^{q-1}$, and define
\[
G_{(q-1)}:=\{s\in G\mid\mbox{the order of $s$ divides $q-1$}\}.
\]
Since $G$ is abelian, for any $h\in\Hom(\Omega_F^t,G)$, we must have $h(\sigma)\in G_{(q-1)}$.

\begin{definition}\label{factorh}
Given $h\in\Hom(\Omega_F^t,G)$, define
\begin{align*}
h^{nr}\in\Hom(\Omega^t_F,G);&\hspace{1em}h^{nr}(\phi):=h(\phi)\mbox{ and }h^{nr}(\sigma):=1;\\
h^{tot}\in\Hom(\Omega^t_F,G);&\hspace{1em}h^{tot}(\phi):=1\mbox{ and }h^{tot}(\sigma):=h(\sigma).
\end{align*}
Clearly $h=h^{nr}h^{tot}$, which we will call the \emph{factorization of $h$ with respect to $\sigma$}. Note that $h^{nr}$ is unramified, and we have $F^{h^{tot}}=F(\pi^{1/|s|})$, where $s=h^{tot}(\sigma)$, by \cite[Proposition 5.4]{M}.
\end{definition}

Now, let $h\in\Hom(\Omega_F^t,G)$ and suppose that $\cO_h=\cO_FG\cdot a$. As in the proof of \cite[Theorem 5.6]{M}, we may decompose the resolvend $\mathbf{r}_G(a)$ as
\[
\mathbf{r}_G(a)=\mathbf{r}_G(a_{nr})\mathbf{r}_G(a_{tot}),
\]
where $\cO_{h^{nr}}=\cO_FG\cdot a_{nr}$ and $\cO_{h^{tot}}=\cO_FG\cdot a_{tot}$. The above is the same decomposition mentioned in (\ref{eq2}). The resolvend $\mathbf{r}_G(a_{nr})$ is \mbox{already characterized by} Proposition~\ref{NBG} (b). In \cite[Proposition 5.4]{M}, McCulloh showed that the element $a_{tot}\in\cO_{h^{tot}}$ may be chosen such that its reduced resolvend $r_G(a_{tot})$ is equal to $\Theta^t(f_s)$ for some $f_s\in\Lambda(FG)^\times$ (recall (\ref{lambda1})) \mbox{that is uniquely deter-} mined by the value $s:=h(\sigma)$.

\begin{definition}\label{primeFlocal}
Given $s\in G_{(q-1)}$, define $f_s=f_{F,s}\in\Lambda(FG)^\times$ by
\[
f_s(t):=\begin{cases}
\pi & \mbox{if }t=s\neq1\\
1 & \mbox{otherwise}.
\end{cases}
\]
Notice that $f_{s}$ indeed preserves the $\Omega_F$-action because all $(q-1)$-st roots of unity are contained in $F$, so elements in $G_{(q-1)}$ are fixed by $\Omega_F$, as is $\pi$. Such a map in $\Lambda(FG)^\times$ is called a \emph{prime $\mathfrak{F}$-element over $F$}. We will write
\[
\mathfrak{F}_F:=\{f_{s}:s\in G_{(q-1)}\}
\]
for the collection of all of prime $\mathfrak{F}$-elements over $F$.
\end{definition}

\subsection{Approximation Theorems} Let $F$ be a number field and define
\begin{equation}\label{primeF}
\mathfrak{F}=\mathfrak{F}_F:=\{f\in J(\Lambda(FG))\mid f_v\in\mathfrak{F}_{F_v}\mbox{ for all }v\in M_F\}.
\end{equation}
We are now ready to state the characterization of the set
\[
R(\cO_FG):=\{\cl(\cO_h):\mbox{tame }h\in\Hom(\Omega_F,G)\}
\]
of all realizable classes in $\mbox{Cl}(\cO_FG)$ that was proved by McCulloh in \mbox{\cite{M}. It will} be helpful to recall (\ref{j}) and Definition~\ref{rag}.

\begin{thm}\label{char1}Let $h\in\Hom(\Omega_F,G)$, say $F_h=FG\cdot b$. Then, we have $h$ is tame if and only if there exists $c\in J(FG)$ such that
\begin{equation}\label{char1'}
rag(c)=\eta(r_G(b))^{-1}u\Theta^t(f)
\end{equation}
for some $u\in U(\mathcal{H}(\cO_FG))$ and $f\in\mathfrak{F}$. Moreover, if (\ref{char1'}) holds, then
\begin{enumerate}[(1)]
\item for all $v\in M_F$, we have $f_v=f_{F_v,s_v}$ for $s_v:=h_v(\sigma_{F_v})$;
\item for all $v\in M_F$, we have $f_v=1$ if and only if $h_v$ is unramified;
\item $j(c)=\cl(\cO_h)$.
\end{enumerate}
\end{thm}
\begin{proof}See \cite[Theorem 6.7]{M}.
\end{proof}

\begin{remark}The decomposition of $rag(c)$ given by (\ref{char1'}) in Theorem~\ref{char1} comes from equation (\ref{eq1}) and the decomposition (\ref{eq2}).
\end{remark}

Theorem~\ref{char1} implies that the set $R(\cO_FG)$ may be characterized as follows. For $c\in J(FG)$, we have $j(c)\in R(\cO_FG)$ if and only if $rag(c)$ is an element of 
\begin{equation}\label{char2}
\eta(\mathcal{H}(FG))U(\mathcal{H}(\cO_FG))\Theta^t(\mathfrak{F}).
\end{equation}
Using two approximations theorems, McCulloh showed in \cite[Theorem 6.17]{M} that the set $\mathfrak{F}$ may be replaced by the group $J(\Lambda(FG))$ in the above. This in turns shows that $R(\cO_FG)$ is a subgroup of $\mbox{Cl}(\cO_FG)$ (see \cite[Corollary 6.20]{M}).

Below, we will state the two approximation theorems, which will be needed in the proof of Theorem~\ref{thm} (a). We will need the following notation. It will be helpful to recall (\ref{lambda1}) and Definition~\ref{cyclotomic}.

\begin{definition}
Let $\mathfrak{m}$ be an ideal in $\cO_F$. For each $v\in M_F$, let
\begin{align*}
U_{\mathfrak{m}}(\cO_{F_v^c})&:=(1+\mathfrak{m}\cO_{F_v^c})\cap(\cO_{F_v^c})^{\times};\\
U'_{\mathfrak{m}}(\Lambda(\cO_{F_v}G))&:=\{g_v\in\Lambda(\cO_{F_v}G)^\times\mid g_v(s)\in U_\mathfrak{m}(\cO_{F_v^c})\mbox{ for all }s\in G\mbox{ with }s\neq 1\}.
\end{align*}
Define
\[
U'_\mathfrak{m}(\Lambda(\cO_FG))
:=\Bigg(\prod_{v\in M_F}U_\mathfrak{m}'(\Lambda(\cO_{F_v}G))\Bigg)\cap J(\Lambda(FG)).
\]
\end{definition}

\begin{definition}\label{gs}
For $g\in J(\Lambda(FG))$ and $s\in G$, define
\[
g_s:=\prod_{v\in M_F}g_v(s)\in\prod_{v\in M_F}(F_v^c)^\times.
\]
\end{definition}

\begin{thm}\label{approx1}Let $\mathfrak{m}$ be an ideal in $\cO_F$ divisible by both $|G|$ and $\exp(G)^2$. Then, we have $\Theta^t(U_\mathfrak{m}'(\Lambda(\cO_FG))\subset U(\mathcal{H}(\cO_FG))$.
\end{thm}
\begin{proof}See \cite[Proposition 6.9]{M}.
\end{proof}

\begin{thm}\label{approx2}Let $g\in J(\Lambda(FG))$ and let $T$ be a finite subset of $M_F$. Then, \mbox{there exists} $f\in\mathfrak{F}$ such that $f_v=1$ for all $v\in T$ and 
\[
g\equiv f\hspace{1cm}(\mbox{mod }\lambda(\Lambda(FG)^\times)U_\mathfrak{m}'(\Lambda(\cO_FG))).
\]
Moreover, we may choose $f\in\mathfrak{F}$ so that for each $s\in G(-1)$ with $s\neq 1$, there exists $\omega\in\Omega_F$ such that $f_{\omega\cdot s}\neq 1$.
\end{thm}
\begin{proof}See \cite[Proposition 6.14]{M}.
\end{proof}

\section{Characterization of $\Sigma$-Realizable Classes}\label{SR}

Throughout this section, let $K/k$ denote a fixed Galois extension of number fields and set $\Sigma:=\Gal(K/k)$. We will also assume that $G$ is abelian and fix a left $\Sigma$-module structure on $G$. We will choose $K^c=\mathbb{Q}^c$ and $k^c=\mathbb{Q}^c$, where $\mathbb{Q}^c$ is a fixed algebraic closure of $\mathbb{Q}$ containing $K$, and also the same compatible set $\{\zeta_n:n\in\mathbb{Z}^+\}$ of primitive roots of unity in $\mathbb{Q}^c$ \mbox{for both $k$ and $K$.} As noted in Remark~\ref{tamesubgp}, we will identify $\Hom(\Omega_K^t,G)$ with the subset of $\Hom(\Omega_K,G)$ consisting of the tame homomorphisms.

\begin{definition}Let $V_k$ denote the set of primes in $M_k$ which are ramified in $K/k$, and let $V_K$ denote the set of primes in $M_K$ lying above those in $V_k$.
\end{definition}

The purpose of this section is to characterize the $\Sigma$-realizable classes (recall (\ref{Srealizable})) coming from the homomorphisms $h\in\Hom(\Omega_K^t,G)^\Sigma$ \mbox{for which $h_v$ is un-} ramified for all $v\in V$, under the hypotheses of Theorem~\ref{thm}. We will do so by refining the characterization of $R(\cO_KG)$ given in Theorem~\ref{char1}. The crucial step is to make suitable choices for the embeddings $i_v:\mathbb{Q}^c\longrightarrow K^c_v$ as well as the uniformizers $\pi_v$ in $K_v$ for $v\in M_K$. We will need the following notation.

\begin{definition}\label{7setup}For each $w\in M_k$, let $i_w:\mathbb{Q}^c\longrightarrow k_w^c$ be the chosen embedding extending the natural embedding $k\longrightarrow k_w$. The \emph{distinguished prime (in $K$) lying above $w$} is the prime $v_w\in M_K$ for which the $v_w$-adic absolute value on $K$ is induced by $i_w$. Moreover, for each $v\in M_K$ lying above $w$, choose an element $\gamma_v\in\Sigma$ such that $v=v_w\circ\gamma_v^{-1}$, and we will choose $\gamma_{v_w}=1$. Choose once and for all a lift $\overline{\gamma_v}\in\Omega_k$ of $\gamma_v$ with  $\overline{\gamma_{v_w}}=1$.
\end{definition}

\subsection{Choices of Embeddings}\label{s:7.1}

\begin{definition}\label{embeddings}Given $v\in M_K$, let $w\in M_k$ be the prime lying \mbox{below $v$ and} notice that the $v$-adic absolute value on $K$ is induced by $i_w\circ\overline{\gamma_v}^{-1}$. \mbox{Via re-} stricting $i_w\circ\overline{\gamma_v}^{-1}$, we then obtain an embedding  $K\longrightarrow k_w^c$ which extends to a continuous embedding $K_v\longrightarrow k_w^c$. We will then lift this to an isomorphism $\varepsilon_{v}^{-1}:K_{v}^c\longrightarrow k_w^c$ and define $i_{v}:\mathbb{Q}^c\longrightarrow K_v^c$ by setting $i_{v}:=\varepsilon_{v}\circ i_w\circ\overline{\gamma_v}^{-1}$, which clearly extends the natural embedding $K\longrightarrow K_{v}$. 
\end{definition}

In summary, for all $v\in M_K$ and $w\in M_k$ with $w$ lying below $v$, the diagram
\[
\begin{tikzcd}[column sep=3cm, row sep=2cm]
 K_{v_w}^c\arrow[pos=0.60, bend left=25]{rr}[font=\normalsize]{\widetilde{\gamma_v}}
  & k_w^c \arrow{l}[swap,font=\normalsize]{\varepsilon_{v_w}}
\arrow{r}[font=\normalsize]{\varepsilon_v}
  & K_v^c \\
\mathbb{Q}^c \arrow[-, double equal sign distance]{r} \arrow[hookrightarrow]{u}[font=\normalsize]{i_{v_w}}
  & \mathbb{Q}^c\arrow{r}[swap,font=\normalsize]{\overline{\gamma_v}} \arrow[hookrightarrow]{u}[font=\normalsize]{i_w}
  & \mathbb{Q}^c \arrow[hookrightarrow]{u}[font=\normalsize]{i_v}
\end{tikzcd}
\]
commutes. Here, we define $\widetilde{\gamma_v}:=\varepsilon_v\circ\varepsilon_{v_w}^{-1}$. Observe that $\widetilde{\gamma_{v_w}}=1$ and we have the relation
\begin{equation}\label{ivrelation}
i_v=\widetilde{\gamma_v}\circ i_{v_w}\circ \overline{\gamma_v}^{-1}.
\end{equation}
Via identifying $\Hom(\Omega_K^t,G)$ with the subset of $\Hom(\Omega_K,G)$ consisting of the tame homomorphisms as in Remark~\ref{tamesubgp}, we have the following result.

\begin{prop}\label{hvcompatible}Let $h\in\Hom(\Omega_K^t,G)^\Sigma$. For all $v\in M_K$ and $w\in M_k$ such that $w$ lies below $v$, we have
\[
 h_v(\widetilde{\gamma_v}\circ\omega\circ\widetilde{\gamma_v}^{-1})=\gamma_v\cdot h_{v_w}(\omega)\hspace{1cm}\mbox{for all $\omega\in\Omega_{K_{v_w}}$}.
\]
\end{prop}
\begin{proof}Let $v\in M_K$ and $w\in M_k$ with $w$ lying below $v$. We have $h_v=h\circ\widetilde{i_v}$ (recall (\ref{iv})) by definition. Using (\ref{ivrelation}), we see that for all $\omega\in\Omega_{K_{v_w}}$, we have
\begin{align*}
h_v(\widetilde{\gamma_v}\circ\omega\circ\widetilde{\gamma_v}^{-1})
&=h(i_v^{-1}\circ\widetilde{\gamma_v}\circ\omega\circ\widetilde{\gamma_v}^{-1}\circ i_v)\\
&=h(\overline{\gamma_v}\circ i_{v_w}^{-1}\circ\omega\circ i_{v_w}\circ\overline{\gamma_v}^{-1})\\
&=\gamma_v\cdot (h\cdot\gamma_v)(i_{v_w}^{-1}\circ\omega\circ i_{v_w})\\
&=\gamma_v\cdot h_{v_w}(\omega),
\end{align*}
where the last equality holds because $h$ is $\Sigma$-invariant. 
\end{proof}

\subsection{Choices of Uniformizers and their Radicals}For each $w\in M_k$, let $\pi_w$ be a chosen uniformizer in $k_w$ and let $\{\pi_w^{1/n}:n\in\mathbb{Z}^+\}$ denote the chosen coherent set of radicals of $\pi_w$ in $k_w^c$ (recall Subsection~\ref{local F sec} above).

\begin{definition}\label{uniformizers}Given $v\in M_K\setminus V_K$, let $w\in M_k$ be the prime lying below $v$. We will choose $\pi_v:=\varepsilon_v(\pi_w)$ to be the uniformizer in $K_v$, and $\pi_v^{1/n}:=\varepsilon_v(\pi_w^{1/n})$ for $n\in\mathbb{Z}^+$ to the be coherent radicals of $\pi_v$ in $K_v^c$. As for $v\in V_K$, we will choose the uniformizer $\pi_v$ in $K_v$ and its radicals arbitrarily.
\end{definition}

\begin{lem}\label{pirel}For all $v\in M_K$ and $w\in M_k$ such that $w$ lies below $v$, we have that $v\notin V_K$ if and only if $v_w\notin V_K$. In particular, if $v\notin V_K$, then  for all $n\in\mathbb{Z}^+$ we have $\pi_v^{1/n}=\widetilde{\gamma_v}(\pi_{v_w}^{1/n})$.
\end{lem}
\begin{proof}Since $K/k$ is Galois, it is clear that $v\notin V_K$ if and only if $v_w\notin V_K$. Using the equality $\widetilde{\gamma_v}=\varepsilon_v\circ\varepsilon_{v_w}^{-1}$, we see that if $v\notin V_K$, then
\[
\pi_v^{1/n}=\varepsilon_v(\pi_w^{1/n})=\widetilde{\gamma_v}(\varepsilon_{v_w}(\pi_w^{1/n}))=\widetilde{\gamma_v}(\pi_{v_w}^{1/n})
\]
for all $n\in\mathbb{Z}^+$ by Definition~\ref{uniformizers}.
\end{proof}

Notice that the choices made in Definitions~\ref{embeddings} and~\ref{uniformizers} in turn determine a distinguished topological generator $\sigma_v=\sigma_{K_v}$ of $\Gal(K_v^t/K_v^{nr})$ (recall (\ref{sigma})). In particular, because we chose $\{i_v(\zeta_n):n\in\mathbb{Z}^+\}$ to be the compatible set of primitive roots of unity in $K_v^c$, we have
\begin{equation}\label{sigma'}
\sigma_v(\pi_v^{1/n})=i_v(\zeta_n)\pi_v^{1/n}\hspace{1cm}\mbox{for }(n,p)=1,
\end{equation}
where $p$ is the rational prime lying below $v$. By abuse of notation, we will also use $\sigma_v$ to denote some chosen lift of $\sigma_v$ in $\Omega_{K_v}$. Via identifying $\Hom(\Omega_K^t,G)$ with the subset of $\Hom(\Omega_K,G)$ consisting of the tame homomorphisms as in Remark~\ref{tamesubgp}, we then have the following result.

\begin{prop}\label{hsigma}Let $h\in\Hom(\Omega_K^t,G)^\Sigma$. For all $v\in M_K\setminus V_K$ and $w\in M_k$ such that $w$ lies below $v$, if $\zeta_{e_v}\in k$ where $e_v:=|h_v(\sigma_v)|$, then \[
h_v(\sigma_v)=\gamma_v\cdot h_{v_w}(\sigma_{v_w}).
\]
\end{prop}
\begin{proof}Let $v\in M_K\setminus V_K$ and $w\in M_k$ with $w$ lying below $v$. We already know from Proposition~\ref{hvcompatible} that $h_v(\widetilde{\gamma_v}\circ\sigma_{v_w}\circ\widetilde{\gamma_v}^{-1})=\gamma_v\cdot h_{v_w}(\sigma_{v_w})$. Thus, it suffices to show that $h_v(\widetilde{\gamma_v}\circ\sigma_{v_w}\circ\widetilde{\gamma_v}^{-1})=h_v(\sigma_v)$, or equivalently, that $\widetilde{\gamma_v}\circ\sigma_{v_w}\circ\widetilde{\gamma_v}^{-1}$ and $\sigma_v$ have the same action on the fixed field $L:=K_v^{h_v}$ of $\ker(h_v)$.

Let $h_v=h_v^{nr}h_v^{tot}$ be the factorization of $h_v$ with respect to $\sigma_v$ (recall Definition~\ref{factorh}). Set $L^{nr}:=K_v^{h_v^{nr}}$ and $L^{tot}:=K_v^{h_v^{tot}}$. It is clear that $L\subset L^{nr}L^{tot}$ and that both $\widetilde{\gamma_v}\circ\sigma_{v_w}\circ\widetilde{\gamma_v}^{-1}$ and $\sigma_v$ act as the identity on $L^{nr}$ because $L^{nr}/K_v$ is unramified. We also have $L^{tot}=K_v(\pi_v^{1/e_v})$ by \cite[Proposition 5.4]{M}. Hence, it remains to show that
\[
(\widetilde{\gamma_v}\circ\sigma_{v_w}\circ\widetilde{\gamma_v}^{-1})(\pi_v^{1/e_v})=\sigma_v(\pi_v^{1/e_v}).
\]
But $\pi_v^{1/e_v}=\widetilde{\gamma_v}(\pi_{v_w}^{1/e_v})$ by Lemma~\ref{pirel} since $v\notin V_K$. Using (\ref{sigma'}), we obtain
\begin{align*}
(\widetilde{\gamma_v}\circ\sigma_{v_w}\circ\widetilde{\gamma_v}^{-1})(\pi_v^{1/e_v})
&=\widetilde{\gamma_v}(i_{v_w}(\zeta_{e_v})\pi_{v_w}^{1/e_v})\\
&=(\widetilde{\gamma_v}\circ i_{v_w})(\zeta_{e_v})\pi_v^{1/e_v}\\
&=(\widetilde{\gamma_v}\circ i_{v_w}\circ\overline{\gamma_v}^{-1})(\zeta_{e_v})\pi_v^{1/e_v}\\
&=i_v(\zeta_{e_v})\pi_v^{1/e_v}\\
&=\sigma_v(\pi_v^{1/e_v}),
\end{align*}
where $\overline{\gamma_v}^{-1}(\zeta_{e_v})=\zeta_{e_v}$ because $\zeta_{e_v}\in k$ by hypothesis and $i_v=\widetilde{\gamma_v}\circ i_{v_w}\circ \overline{\gamma_v}^{-1}$ by  (\ref{ivrelation}). So, indeed $\widetilde{\gamma_v}\circ\sigma_{v_w}\circ\widetilde{\gamma_v}^{-1}$ and $\sigma_v$ have the same action on $L$. This proves the claim.
\end{proof}

\subsection{Embeddings of Groups of Ideles}\label{s:7.2} In this subsection, assume further that $k$ contains all $\exp(G)$-th roots of unity. Recall (\ref{lambda1}) and Definition~\ref{cyclotomic}, and observe that $\Lambda(FG)=\Map(G,F)$ for $F\in\{k,K,k_w,K_v\}$, where $w\in M_k$ and $v\in M_K$. It will also be helpful to recall Definitions~\ref{rag} \mbox{and~\ref{Theta}. The iso-} morphisms $\varepsilon_v$ for $v\in M_K$ in Definition~\ref{embeddings} then induce the following embeddings of groups of ideles.

\begin{definition}Define $\nu: J(\Lambda(kG))\longrightarrow J(\Lambda(KG))$ by setting
\[
\nu(g)_v:=\varepsilon_v\circ g_w\hspace{1cm}\mbox{for each }v\in M_K,
\]
where $w\in M_k$ is the prime lying below $v$. 

Similarly, define $\mu:J(\mathcal{H}(kG))\longrightarrow J(\mathcal{H}(KG))$ by setting
\[
\mu((r_G(a))_v:=r_G(\varepsilon_v\circ a_w)\hspace{1cm}
\mbox{for each }v\in M_K,
\]
where $w\in M_k$ is the prime lying below $v$ and $a_w\in\Map(G,k_w^c)$ is an element such that $r_G(a_w)=r_G(a)_w$. Notice that the definition of $\mu$ does not require that $k$ contains all $\exp(G)$-th roots of unity.
\end{definition}

First, we will prove two basic properties of the map $\nu$. Recall Definition~\ref{primeFlocal} and (\ref{primeF}), and note that the choices of the uniformizers $\pi_w$ in $k_w$ for $w\in M_k$ determine a subset $\mathfrak{F}_k$ of $J(\Lambda(kG))$. Similarly, the choices of the uniformizers $\pi_v$ in $K_v$ for $v\in M_K$ in Definition~\ref{uniformizers} determine a subset $\mathfrak{F}_K$ of $J(\Lambda(KG))$.

\begin{prop}\label{nuF}Let $f\in\mathfrak{F}_k$ and write $f_w=f_{k_w,s_w}$ for each $w\in M_k$. For all $v\in M_K\setminus V_K$ and $w\in M_k$ such that $w$ lies below $v$, we have $\nu(f)_v=f_{K_v,s_w}$. In particular, if $f_w=1$ for all $w\in V_k$, then $\nu(f)\in\mathfrak{F}_K$.
\end{prop}
\begin{proof}Let $v\in M_K\setminus V_K$ and $w\in M_k$ with $w$ lying below $v$. Let $q_v$ and $q_w$ denote the orders of the residue fields of $K_v$ and $k_w$, respectively. The order of $s_w$ divides $q_w-1$ by definition and hence divides $q_v-1$. Because $v\notin V_K$, we have $\pi_v=\varepsilon_v(\pi_w)$ by definition and it is clear that $\nu(f)_v=f_{K_v,s_w}$. We then see that $\nu(f)_v\in\mathfrak{F}_{K_v}$. If $f_w=1$ for all $w\in V_k$, then $\nu(f)_v=1$ \mbox{lies in $\mathfrak{F}_{K_v}$} for all $v\in V_K$ as well. We then deduce that $\nu(f)\in\mathfrak{F}_K$ in this case.
\end{proof}

\begin{prop}\label{nuF'}Let $f\in\mathfrak{F}_K$ and write $f_v=f_{K_v,s_v}$ for each $v\in M_K$. If
\begin{enumerate}[(1)]
\item $s_v=1$ for all $v\in V_K$; and
\item $s_v=s_{v_w}$ for all $v\in M_K$ and $w\in M_k$ such that $w$ lies below $v$,
\end{enumerate}
then we have $f=\nu(g)$ for some $g\in J(\Lambda(kG))$.
\end{prop}
\begin{proof}
For each $w\in M_k$, define $g_w\in\Lambda(k_wG)^\times=\Map(G,k_w^\times)$ by
\[
g_w(s):=
\begin{cases}
\pi_w &\mbox{if }s=s_{v_w}\neq 1\\
1 &\mbox{otherwise}.
\end{cases}
\]
Note that $g:=(g_w)_w\in J(\Lambda(kG))$ because $f\in J(\Lambda(KG))$ implies that $s_{v}=1$ for all but finitely many $v\in M_K$. To prove that $f=\nu(g)$, let $v\in M_K$ and let $w\in M_k$ be the prime lying below $v$. If $s_{v_w}\neq 1$, then $s_v\neq1$ also by (2) and so $v\notin V_K$ by (1). We then obtain  $\nu(g)_v=f_{K_v,s_v}$ since $\pi_v=\varepsilon_v(\pi_w)$ by definition and $s_v=s_{v_w}$ by (2). If $s_{v_w}=1$, then $s_v=1$ by (2) and so $\nu(g)_v=1=f_{K_v,s_v}$. This shows that $f=\nu(g)$ and so $f\in\nu(J(\Lambda(kG)))$, as claimed.
\end{proof}

Next, we will show that certain diagrams involving $\nu$ and $\mu$ commute.

\begin{prop}\label{commutenu}The diagram
\[
\begin{tikzpicture}[baseline=(current bounding box.center)]
\node at (0,2.5) [name=13] {$\Lambda(kG)^\times$};
\node at (5.5,2.5) [name=14] {$J(\Lambda(kG))$};
\node at (0,0) [name=23] {$\Lambda(KG)^\times$};
\node at (5.5,0) [name=24] {$J(\Lambda(KG))$};
\path[->, font=\normalsize]
(13) edge node[auto]{$\lambda_k$} (14)
(23) edge node[below]{$\lambda_K$} (24)
(14) edge node[right]{$\nu$} (24)
(13) edge node[left]{$\iota_\Lambda$} (23);
\end{tikzpicture}
\]
commutes, where $\iota_\Lambda$ is the map induced by the natural inclusion $k\longrightarrow K$.
\end{prop}
\begin{proof}Recall that $\lambda_k$ and $\lambda_K$ denote the diagonal maps. Now, let $g\in\Lambda(kG)^\times$ be given. Also, let $v\in M_K$ and let $w\in M_k$ be the prime lying below $v$. Then

\begin{align*}
(\nu\circ\lambda_k)(g)_v
&=\varepsilon_v\circ i_w\circ g\\
&=i_v\circ\overline{\gamma_v}\circ g\\
&=i_v\circ g\\
&=(\lambda_K\circ\iota_\Lambda)(g)_v,
\end{align*}
where $\varepsilon_v\circ i_w=i_v\circ\overline{\gamma_v}$ holds by Definition~\ref{embeddings} and $\overline{\gamma_v}\circ g=g$ because $g$ takes values in $k$. So, we have $\nu\circ\lambda_k=\lambda_K\circ\iota_\Lambda$ and the diagram commutes.
\end{proof}

\begin{prop}\label{commutetheta}The diagram
\[
\begin{tikzpicture}[baseline=(current bounding box.center)]
\node at (0,2.5) [name=13] {$J(\Lambda(kG))$};
\node at (5.5,2.5) [name=14] {$J(\Lambda(KG))$};
\node at (0,0) [name=23] {$J(\mathcal{H}(kG))$};
\node at (5.5,0) [name=24] {$J(\mathcal{H}(KG))$};
\path[->, font=\normalsize]
(13) edge node[auto]{$\nu$} (14)
(23) edge node[below]{$\mu$} (24)
(14) edge node[right]{$\Theta^t_{K}$} (24)
(13) edge node[left]{$\Theta^t_{k}$} (23);
\end{tikzpicture}
\]
commutes.
\end{prop}
\begin{proof}Let $g\in J(\Lambda(kG))$ be given. Also, let $v\in M_K$ and let $w\in M_k$ be the \par\noindent prime lying below $v$. On one hand, we have
\begin{equation}\label{exp1}
(\Theta_K^t\circ\nu)(g)_v=\Theta^t_{K}(\varepsilon_v\circ g_w).
\end{equation}
On the other hand, let $r_G(a_w)\in\mathcal{H}(k_wG)$ be such that $\Theta^t_{k}(g_w)=r_G(a_w)$, so
\begin{equation}\label{exp2}
(\mu\circ\Theta^t_{k})(g)_v=r_G(\varepsilon_v\circ a_w).
\end{equation}
Let $\widehat{G}_w$ denote the group of irreducible $k_w^c$-valued characters on $G$ and recall the notation $A_{\widehat{G}_w}$ from (\ref{SG}). By the identification
\[
\mathcal{H}(k_wG)=\Hom_{\Omega_{k_w}}(A_{\widehat{G}_w},(k_w^c)^\times)
\]
in (\ref{iden}), we have that $r_G(a_w)(\psi)=\Theta^t_{k}(g)(\psi)$ for all $\psi\in A_{\widehat{G}_w}$. Similarly, let $\widehat{G}_v$ denote the group of irreducible $K_v^c$-valued characters on $G$.  We will use the identification
\[
\mathcal{H}(K_vG)=\Hom_{\Omega_{K_v}}(A_{\widehat{G}_v},(K_v^c)^\times)
\]
given in (\ref{iden}) to show that the expressions in (\ref{exp1}) and (\ref{exp2}) are equal.

To that end, let $\psi\in A_{\widehat{G}_v}$ and write $\psi=\sum_\chi n_\chi\chi$. Define 
\[
\varepsilon_v^{-1}\circ\psi:=\sum_\chi n_\chi(\varepsilon_v^{-1}\circ\chi),
\]
which clearly lies in $A_{\widehat{G}_w}$. Since $r_G(a_w)=\Theta^t_{k}(g_w)$, we then deduce that
\begin{align}\label{computation}
r_G(\varepsilon_v\circ a_w)(\psi)
&=\varepsilon_v(r_G(a_w)(\varepsilon_v^{-1}\circ\psi))\\\notag
&=\varepsilon_v\Bigg(\prod_{s\in G}g_w(s)^{\langle\varepsilon_v^{-1}\circ\psi,s\rangle}\Bigg)\\\notag
&=\prod_{s\in G}(\varepsilon_v\circ g_w)(s)^{\langle\psi,s\rangle}\\\notag
&=\Theta^t_{K}(\varepsilon_v\circ g_w)(\psi),
\end{align}
where the third equality holds because $\langle\varepsilon_v^{-1}\circ\psi,s\rangle=\langle\psi,s\rangle$ for all $s\in G$. To see why, observe that clearly it suffices to show that $\langle\varepsilon_v^{-1}\circ\chi,s\rangle=\langle\chi,s\rangle$ holds for all $\chi\in\widehat{G}_v$ and $s\in G$. It will be helpful to recall that we chose the same compatible set $\{\zeta_n:n\in\mathbb{Z}^+\}$ of roots of unity in $\mathbb{Q}^c$ for $k$ and $K$. Also, we chose $\{i_v(\zeta_n):n\in\mathbb{Z}^+\}$ and $\{i_w(\zeta_n):n\in\mathbb{Z}^+\}$ to be the compatible sets of roots of unity in $K_v^c$ and $k_w^c$, respectively.

Now, let $\chi\in\widehat{G}_v$ and $s\in G$ be given. Let $\upsilon=\upsilon(\chi,s)$ be as in Definition~\ref{Stickel}. Then, we have $\chi(s)=i_v(\zeta_{|s|})^{\upsilon}$ and $\langle\chi,s\rangle=\upsilon/|s|$. Observe that
\begin{align*}
(\varepsilon_v^{-1}\circ\chi)(s)
&=(\varepsilon_v^{-1}\circ i_v)(\zeta_{|s|})^\upsilon\\
&=(i_w\circ\overline{\gamma_v}^{-1})(\zeta_{|s|})^\upsilon\\
&=i_w(\zeta_{|s|})^\upsilon,
\end{align*}
where $\varepsilon_v^{-1}\circ i_v=i_w\circ\overline{\gamma_v}^{-1}$  by Definition~\ref{embeddings} and $\overline{\gamma_v}^{-1}(\zeta_{|s|})=\zeta_{|s|}$ \mbox{because $k$} contains all $\exp(G)$-th roots of unity. Then, by Definition~\ref{Stickel}, we deduce that $\langle\varepsilon_v^{-1}\circ\chi,s\rangle=\upsilon/|s|$ as well. This shows that the third equality in (\ref{computation}) indeed holds, whence (\ref{exp1}) and (\ref{exp2}) are equal. It follows that $\Theta^t_{K}\circ\nu=\mu\circ\Theta^t_{k}$ and so the diagram commutes.
\end{proof}

\subsection{Preliminary Definitions}\label{s:7.3}In this subsection, we will let $\Sigma$ act trivially on $G$ on the left. Then, the left $\Omega_k$-action on $G$ induced by the natural quot- ient map $\Omega_k\longrightarrow\Sigma$ and the given left $\Sigma$-action on $G$ is trivial, which agrees with our definition in Subsection~\ref{notation} above. So, analogous to (\ref{toprow}), from the Hochschild-Serre spectral sequence associated to the group extension
\[
\begin{tikzcd}[column sep=1cm, row sep=1.5cm]
1 \arrow{r} &
\Omega_K \arrow{r} &
\Omega_k \arrow{r} &
\Sigma\arrow{r} &
1,
\end{tikzcd}
\]
we obtain an exact sequence
\begin{equation}\label{toprow'}
\begin{tikzcd}[column sep=1.5cm]
\Hom(\Omega_k,G)\arrow{r}{\mbox{res}} & \Hom(\Omega_K,G)^\Sigma\arrow{r}[font=\normalsize]{tr} & H^2(\Sigma,G).
\end{tikzcd}
\end{equation}
Here $\mbox{res}$ is given by restriction. The $\Sigma$-action on $\Hom(\Omega_K,G)$ and the \emph{transgression map} $tr$ are defined as in Definition~\ref{tr}. To be precise, for each $\gamma\in\Sigma$, choose and fix a lift $\overline{\gamma}\in\Omega_k$ of $\gamma$.

\begin{definition}\label{tr'}Given $h\in\Hom(\Omega_K,G)$ and $\gamma\in\Sigma$, define
\[
(h\cdot\gamma)(\omega):= \gamma^{-1}\cdot h(\lift{\gamma}\omega\lift{\gamma}^{-1})\hspace{1cm}\mbox{for all $\omega\in\Omega_K$}.
\]
The \emph{transgression map} $tr:\Hom(\Omega_K,G)^\Sigma\longrightarrow H^2(\Sigma,G)$ is defined by
\[
tr(h):=[(\gamma,\delta)\mapsto h((\lift{\gamma})(\lift{\delta})(\lift{\gamma\delta})^{-1})],
\]
where $[-]$ denotes the cohomology class. These definitions do not depend on the choice of the lifts $\lift{\gamma}$ for $\gamma\in\Sigma$.
\end{definition}

If we identify $\Hom(\Omega_K^t,G)$ as the subset of $\Hom(\Omega_K,G)$ consisting of the tame homomorphisms via Remark~\ref{tamesubgp}, then the $\Sigma$-action on $\Hom(\Omega_K^t,G)$ and the transgression map on $\Hom(\Omega_K^t,G)^\Sigma$ induced by Definition~\ref{tr'}  agree with those given in Definition~\ref{tr}. In particular, the identical notation does not cause any confusion.

\begin{definition}Define (recall Definition~\ref{ha})
\begin{align*}
\mathcal{H}_\Sigma(KG)&:=\{r_G(a)\in\mathcal{H}(KG)\mid h_a\in \Hom(\Omega_K,G)^\Sigma\};\\
\mathcal{H}_s(KG)&:=\{r_G(a)\in\mathcal{H}(KG)\mid h_a\in \Hom(\Omega_K,G)^\Sigma)\mbox{ and }tr(h_a)=1\}.
\end{align*}
It is clear that both of the sets above are subgroups of $\mathcal{H}(KG)$. 
\end{definition}

\begin{prop}\label{Hs}Assume that $k$ contains all $\exp(G)$-th roots of unity. Then, we have (recall Definitions~\ref{rag} and~\ref{Theta})
\[
(\Theta^t_{K}\circ\nu)(\lambda_k(\Lambda(kG)^\times))\subset\eta(\mathcal{H}_s(KG)).
\]
\end{prop}
\begin{proof}Because $k$ contains all $\exp(G)$-th roots of unity, the map $\nu$ is defined and results from Subsection~\ref{s:7.2} apply. Now, let $g\in\Lambda(kG)^\times$ be given. Recall that $\iota_\Lambda:\Lambda(kG)^\times\longrightarrow\lambda(KG)^\times$ denotes the map induced by the natural inclusion $k\longrightarrow K$. Then, we have
\[
(\Theta_{K}^t\circ\nu)(\lambda_k(g))
=(\Theta_{K}^t\circ\lambda_K)(\iota_\Lambda(g))
=(\eta\circ\Theta^t_{K})(\iota_\Lambda(g)),
\]
where $\nu\circ\lambda_k=\lambda_K\circ\iota_\Lambda$ by Proposition~\ref{commutenu} and $\Theta^t_{K}\circ\lambda_K=\eta\circ\Theta^t_{K}$ because diagram (\ref{Theta'}) commutes. It remains to show that $\Theta^t_{K}(\iota_\Lambda(g))\in\mathcal{H}_s(KG)$. 

Recall that $\mathcal{H}(kG)=((\mathbb{Q}^cG)^\times/G)^{\Omega_k}$ and $\mathcal{H}(KG)=((\mathbb{Q}^cG)^\times/G)^{\Omega_K}$ by definition. Using the identification in (\ref{iden}), we have a commutative diagram
\[
\begin{tikzpicture}[baseline=(current bounding box.center)]
\node at (0,2.5) [name=13] {$\Hom_{\Omega_k}(A_{\widehat{G}},(\mathbb{Q}^c)^\times)$};
\node at (5.5,2.5) [name=14] {$\Hom_{\Omega_K}(A_{\widehat{G}},(\mathbb{Q}^c)^\times)$};
\node at (0,0) [name=23] {$\mathcal{H}(kG)$};
\node at (5.5,0) [name=24] {$\mathcal{H}(KG)$};
\path[->]
(23) edge node[below]{$\iota_\mathcal{H}$} (24)
(13) edge node[auto] {} (14);
\draw[double equal sign distance]
(13) -- (23)
(14) -- (24);
\end{tikzpicture},
\]
where the two horizontal maps are the obvious inclusion maps induced by the inclusion $\Omega_K\subset\Omega_k$. If $\Theta^t_{k}(g)=r_G(a)$, then clearly $\Theta^t_{K}(\iota_\Lambda(g))=\iota_\mathcal{H}(r_G(a))$.  In particular, the homomorphism $h$ associated to $\Theta^t_{K}(\iota_\Lambda(g))$ is equal to $res(h_a)$. Since (\ref{toprow'}) is exact, this shows that $tr(h)=1$ and so $\Theta^t_{K}(\iota_\Lambda(g))\in\mathcal{H}_s(KG)$. This proves the claim.
\end{proof}

\subsection{Proof of Theorem~\ref{thm} (a)}\label{s:7.4}

\begin{proof}Let $\uprho_\Sigma$ denote the composition of the homomorphism $rag$ from Definition~\ref{rag} followed by the natural quotient map
\[
J(\mathcal{H}(KG))\longrightarrow\frac{J(\mathcal{H}(KG))}{\eta(\mathcal{H}_\Sigma(KG))U(\mathcal{H}(\cO_KG))(\Theta^t_{K}\circ\nu)(J(\Lambda(kG)))}.
\]
We will show that $R_\Sigma(\cO_KG)_V $ is a subgroup of $\mbox{Cl}(\cO_KG)$ by showing that
\begin{equation}\label{jschar}
j^{-1}(R_\Sigma(\cO_KG)_V )=\ker(\upvarrho_\Sigma),
\end{equation}
or equivalently, for any $c\in J(KG)$, we have $j(c)\in R_\Sigma(\cO_KG)_V $ if and only if
\begin{equation}\label{charV}
rag(c)\in\eta(\mathcal{H}_\Sigma(KG))U(\mathcal{H}(\cO_KG))(\Theta^t_{K}\circ\nu)(J(\Lambda(kG))).
\end{equation}

To that end, let $c\in J(KG)$ be given. First, assume that (\ref{charV}) holds, so
\begin{equation}\label{charV'}
rag(c)=\eta(r_G(b))^{-1}u(\Theta^t_{K}\circ\nu)(g)
\end{equation}
for some $r_G(b)\in\mathcal{H}_\Sigma(KG)$, $u\in U(\mathcal{H}(\cO_KG))$, and $g\in J(\Lambda(kG))$. Let $\mathfrak{m}$ be \mbox{an ideal in $\cO_k$.} Then, by Theorem~\ref{approx2}, there exists $f\in\mathfrak{F}_{k}$ such that $f_w=1$ for all primes $w\in M_k$ which are ramified in $K/k$ and
\[
g\equiv f\hspace{1cm}(\mbox{mod }\lambda_k(\Lambda(kG)^\times)U_\mathfrak{m}'(\Lambda(\cO_kG))).
\]
Choosing $\mathfrak{m}$ to be divisible by $|G|$ and $\exp(G)^2$, the above then yields
\[
\Theta^t_{k}(g)\equiv\Theta^t_{k}(f)\hspace{1cm}(\mbox{mod }\Theta^t_{k}(\lambda_k(\Lambda(kG)^\times))U(\mathcal{H}(\cO_kG)))
\]
by Theorem~\ref{approx1}. Since $\mu\circ\Theta^t_{k}=\Theta^t_{K}\circ\nu$ by Proposition~\ref{commutetheta}, applying $\mu$ to the above equation and using  by Proposition~\ref{Hs}, we then obtain
\begin{equation}\label{f=g'}
(\Theta^t_{K}\circ\nu)(g)\equiv(\Theta^t_{K}\circ\nu)(f)
\hspace{1cm}(\mbox{mod }\eta(\mathcal{H}_s(KG))U(\mathcal{H}(\cO_KG))).
\end{equation}
Thus, by changing $b$ and $u$ in (\ref{charV'}) if necessary, we may assume that $g=f$. Notice that $\nu(f)_v=1$ for all $v\in V$ and that $\nu(f)\in\mathfrak{F}_{K}$ by Proposition~\ref{nuF}. Hence, if $h:=h_b$ is the homomorphism associated to $r_G(b)$, then $h$ is tame with $h_v$ unramified for all $v\in V$ and $j(c)=\cl(\cO_h)$ by Theorem~\ref{char1}. Since $r_G(b)\in\mathcal{H}_\Sigma(KG)$, we know that $h$ is $\Sigma$-invariant, and so $j(c)\in R_\Sigma(\cO_KG)_V $.

Conversely, assume that $j(c)=\cl(\cO_h)$ for some $h\in \Hom(\Omega_K^t,G)^\Sigma_V$ and let $K_h=KG\cdot b$. Then, we know from Theorem~\ref{char1} that there exists $c'\in J(KG)$ with $j(c')=\cl(\cO_h)$ such that
\begin{equation}\label{ragc'}
rag(c')=\eta(r_G(b))^{-1}u\Theta^t_{K}(f')
\end{equation}
for some $u\in U(\mathcal{H}(\cO_KG))$ and $f'\in\mathfrak{F}_K$. Moreover, for each $v\in M_K$, we have $f_v'=f_{K_v,s_v}'$ for $s_v=h_v(\sigma_{K_v})$, and $s_v=1$ if $v\in V$. Since $\Sigma$ \mbox{acts trivially on $G$,} by Proposition~\ref{hsigma}, we have $s_v=s_{v_w}$ for all $v\in M_K$ and $w\in M_k$ with $w$ lying below $v$ (recall Definition~\ref{7setup}). Proposition~\ref{nuF'} then implies that $f'=\nu(g)$ for some $g\in J(\Lambda(kG))$. Since $j(c)=\cl(\cO_h)$ also, by Theorem~\ref{isoLFCG}, we have
\[
c\equiv c'\hspace{1cm}(\mbox{mod }\partial((KG)^\times)U(\cO_KG)).
\]
Clearly $rag((KG)^\times)\subset\mathcal{H}_s(KG)$. We may then write (\ref{ragc'}) as
\begin{equation}\label{charV''}
rag(c)=\eta(r_G(b)r_G(b'))^{-1}uu'(\Theta^t_{K}\circ\nu)(g)
\end{equation}
for some $r_G(b')\in\mathcal{H}_s(KG)$ and $u'\in\mathcal{H}(\cO_KG)$. Notice that $r_G(b)\in\mathcal{H}_\Sigma(KG)$ because $h$ is $\Sigma$-invariant, and so (\ref{charV}) holds. This proves (\ref{jschar}). It remains to show the existence of $h'\in\Hom(\Omega_K^t,G)_V^\Sigma$ satisfying (1) to (4).

Let $T$ be a finite set of primes in $\cO_K$. First, notice that the same discussion following (\ref{charV'}) shows that there exists $f\in\mathfrak{F}_k$ such that (\ref{f=g'}) holds. Thus, by changing $b'$ and $u'$ in (\ref{charV''}) if necessary, we may assume that $g=f$. By Theorem~\ref{approx2}, we may also assume that $f_w=1$ for all $w\in M_k$ lying below the primes in $V\cup T$, and that $f_s\neq 1$ for all $s\in G$ with $s\neq 1$ (notice that $\Omega_k$ acts trivially on $G(-1)$ because $k$ contains all $\exp(G)$-th roots of unity). By Proposition~\ref{nuF}, we then have $\nu(f)_v=1$ for all $v\in V\cup T$ and $\nu(f)\in\mathfrak{F}_K$.

Now, let $h'$ denote the homomorphism associated to $r_G(b)r_G(b')$. By (\ref{charV''}) and Theorem~\ref{char1}, we see that $h'$ is tame \mbox{with $h'_v$ unramified for all $v\in V\cup T$} and that $j(c)=\cl(\cO_h)$, whence (2) and (3) hold. Because $r_G(b')\in\mathcal{H}_s(KG)$ and $h=h_b$ is $\Sigma$-invariant, it is clear that $h'\in\Hom(\Omega_K^t,G)^\Sigma_V$ and we have $tr(h')=tr(h)$, so (4) holds as well. Finally, for each $s\in G$ with $s\neq1$, we have $f_s\neq 1$ by choice and so $f_w=f_{k_w,s}$ for some $w\in M_k$. We then see \mbox{that $\nu(f)_v=f_{K_v,s}$ by} Proposition~\ref{nuF} and so $h_v'(\sigma_{K_v})=s$ by Theorem~\ref{char1}. This means that $h'$ is surjective and so $K_{h'}$ is a field, as claimed in (1).

Because $\gal$ is weakly multiplicative (recall (\ref{galweak})), what we have proved above implies that $R_s(\cO_KG)_V$ is closed under multiplication. Since $\Cl(\cO_KG)$ is finite, it follows that $R_s(\cO_KG)_V$ is a subgroup of $\mbox{Cl}(\cO_KG)$ as well. This completes the proof of the theorem.
\end{proof}

\subsection{The Quotient $R_\Sigma(\cO_KG)_V/R_s(\cO_KG)_V$}\label{s:7.4} In what follows, assume all of the hypotheses stated in Theorem~\ref{thm}. The sets $R_\Sigma(\cO_KG)_V $ and $R_s(\cO_KG)_V $ are then subgroups of $\mbox{Cl}(\cO_KG)$ by Theorem~\ref{thm} (a). We are \mbox{interested in the} group structure of their quotient $R_\Sigma(\cO_KG)_V /R_s(\cO_KG)_V$ and its relation to that of $tr(\Hom(\Omega_K^t,G)_V^\Sigma)$. The proposition below is essentially a corollary of Theorem~\ref{thm} (a), and will in turn allow us to prove Theorem~\ref{thm} (b).

\begin{prop}\label{Vprop}Let $h,h_1,h_2\in\Hom(\Omega_K^t,G)^\Sigma_V$. 
\begin{enumerate}[(a)]
\item $\cl(\cO_{h_1})\cl(\cO_{h_2})=\cl(\cO_{h_1h_2h_s})$ for some $h_s\in\Hom(\Omega_K^t,G)^\Sigma_V$ with $tr(h_s)=1$.
\item $\cl(\cO_h)\cl(\cO_{h^{-1}})\equiv1$ $(\mbox{mod }R_s(\cO_KG)_V )$.
\item If $tr(h_1)=tr(h_2)$, then $\cl(\cO_{h_1})\equiv\cl(\cO_{h_2})$ (mod $R_s(\cO_KG)_V$).
\end{enumerate}
\end{prop}
\begin{proof}By Theorem~\ref{thm} (a), there exists $h_2'\in\Hom(\Omega_K^t,G)^\Sigma_V$ such that
\[
\cl(\cO_{h_2'})=\cl(\cO_{h_2}),
\hspace{0.75cm}
tr(h_2')=tr(h_2),
\hspace{0.75cm}d(h_2')\cap d(h_1)=\emptyset
\]
(recall (\ref{d(h)})). Since $\gal$ is weakly multiplicative (recall (\ref{galweak})), we deduce that
\[
\cl(\cO_{h_1})\cl(\cO_{h_2})
=\cl(\cO_{h_1h_2'})
=\cl(\cO_{h_1h_2h_s}),
\]
where $h_s:=h_2^{-1}h_2'$ and clearly $tr(h_s)=1$. This shows that (a) holds, and (b) follows from applying (a) to $h_1=h$ and $h_2=h^{-1}$. Now, observe that (a) and (b) together imply that
\[
\cl(\cO_{h_1})\cl(\cO_{h_2})^{-1}
\equiv\cl(\cO_{h_1})\cl(\cO_{h_2^{-1}})
\equiv \cl(\cO_{h_1h_2^{-1}h_s})\hspace{0.75cm}(\mbox{mod }R_s(\cO_KG)_V )
\]
for some $h_s\in\Hom(\Omega_K^t,G)_V^\Sigma$ with $tr(h_s)=1$. Hence, if $tr(h_1)=tr(h_2)$, then $tr(h_1h_2^{-1}h_s)=1$ and we deduce that $\cl(\cO_{h_1})\equiv\cl(\cO_{h_2})$ (mod $R_s(\cO_KG)_V $), as desired.
\end{proof}

\subsection{Proof of Theorem~\ref{thm} (b)}\label{s:7.5}
\begin{proof}The map $\upphi$ is well-defined by Proposition~\ref{Vprop} (c). To show that it is in fact a homomorphism, let $h_1,h_2\in\Hom(\Omega_K^t,G)_V^\Sigma$ be given. Moreover, let $h_s\in \Hom(\Omega_K^t,G)_V^\Sigma$ be such that $tr(h_s)=1$ and $\cl(\cO_{h_1})\cl(\cO_{h_2})=\cl(\cO_{h_1h_2h_s})$, which exists by Proposition~\ref{Vprop} (a). Then, we have
\[
\upphi(tr(h_1))\upphi(tr(h_2))
=\upphi(tr(h_1h_2h_s))
=\upphi(tr(h_1h_2))
\]
and so $\upphi$ is indeed a homomorphism. This proves the first claim.

To prove the second claim, let $h\in\Hom(\Omega_K^t,G)_V^\Sigma$ be such that $\upphi(tr(h))=1$, that is $\cl(\cO_h)\in R_s(\cO_KG)_V$. Since $R_s(\cO_KG)_V$ is a subgroup of $\Cl(\cO_KG)$ by Theorem~\ref{thm} (a), we have $\cl(\cO_h)^{-1}=\cl(\cO_{h_s})$ for some $h_s\in\Hom(\Omega_K^t,G)_V^\Sigma$ with $tr(h_s)=1$. In particular, we may assume that $d(h_s)\cap d(h)=\emptyset$ (recall (\ref{d(h)})). Since $\gal$ is weakly multiplicative (recall (\ref{galweak})), we then deduce that
\[
1=\cl(\cO_h)\cl(\cO_{h_s})=\cl(\cO_{hh_s}).
\]
Now, recall Theorem~\ref{commutes}. Because $\xi$ is a homomorphism, the above implies that $(\xi\circ\gal)(hh_s)=1$ and hence $(i^*\circ tr)(hh_s)=1$. If $i^*$ is injective, then this yields $tr(hh_s)=1$ and so $tr(h)=1$. Hence, in this case the map $\upphi$ is injective and in particular is an isomorphism.
\end{proof}

\section{Acknowledgments}

I would like to thank my advisor Professor Adebisi Agboola for bringing this problem to my attention and for reading a draft of this paper.


\end{document}